\documentclass[11pt,reqno]{amsart}

\usepackage[english]{babel}
\usepackage[utf8]{inputenc}
\usepackage[T1]{fontenc}
\usepackage{lmodern} \normalfont 
\DeclareFontShape{T1}{lmr}{bx}{sc} { <-> ssub * cmr/bx/sc }{}
\usepackage{microtype}	

\usepackage{amssymb}
\usepackage{amsmath}
\usepackage{amsthm}
\usepackage{mathtools}
\usepackage{mathrsfs}
\usepackage{accents} 
\usepackage{etoolbox}
\usepackage{siunitx}
\usepackage{dsfont} 
\usepackage{graphicx}
\usepackage{color}
\usepackage[dvipsnames]{xcolor}
\usepackage{tikz}
\usetikzlibrary{calc,positioning,shapes}
\usetikzlibrary{patterns,decorations.pathmorphing,decorations.markings}
\usepackage{pgfplots}
\pgfplotsset{compat=newest}
\usepackage[margin=10pt,font=small,labelfont=bf,labelsep=endash]{caption}
\usepackage{subcaption}

\usepackage{booktabs}
\usepackage{paralist}
\usepackage{soul}

\numberwithin{equation}{section}

\usepackage{algorithm}
\usepackage{algorithmicx}
\usepackage{algpseudocode}

\usepackage{cite}
\usepackage{multirow} 
\usepackage{enumitem} 
\setlist[enumerate]{label=(\roman*)}

\usepackage[colorlinks,citecolor=black,urlcolor=black]{hyperref}
\usepackage[nameinlink,noabbrev]{cleveref}

\theoremstyle{plain}
\newtheorem{theorem}{Theorem}[section]

\newtheorem{remark}[theorem]{Remark}
\newtheorem{definition}[theorem]{Definition}

\newtheorem{example}[theorem]{Example}

\newcommand{\R}{\mathbb{R}}
\newcommand{\C}{\mathbb{C}}

\DeclareMathOperator{\rank}{rank}
\DeclareMathOperator{\Skew}{skew}
\DeclareMathOperator{\Sym}{sym}
\DeclareMathOperator{\Ran}{Ran}
\newcommand{\Real}{\mathsf{Re}}

\newcommand{\dd}{\, \text{d}}

\newcommand{\Xmax}{X_{\mathrm{max}}}
\newcommand{\Xmin}{X_{\mathrm{min}}}
\newcommand{\lineWidth}{1.2}

\definecolor{mycolor1}{rgb}{0.00000,0.44700,0.74100}
\definecolor{mycolor2}{rgb}{0.85000,0.32500,0.09800}
\definecolor{mycolor3}{rgb}{0.92900,0.69400,0.12500}
\definecolor{mycolor4}{rgb}{0.46600,0.67400,0.18800}
\definecolor{mycolor5}{rgb}{0.49400,0.18400,0.55600}

\newcommand{\matlab}{MATLAB\textsuperscript{\textregistered}}

\title{Passivity preserving model reduction via spectral factorization}
\author{Tobias Breiten${}^\dagger$ \and Benjamin Unger${}^\star$}
\address{${}^{\dagger}$  Institute of Mathematics MA\,{}4-4, Technical University Berlin, Stra\ss e des 17.~Juni 136, 10623 Berlin, Germany}
\email{breiten@math.tu-berlin.de}
\address{${}^{\star}$ Stuttgart Center for Simulation Science (SC SimTech), University of Stuttgart, Universit\"{a}tsstr.~32, 70569 Stuttgart, Germany}
\email{benjamin.unger@simtech.uni-stuttgart.de}
\date{\today}
\keywords{port-Hamiltonian systems; structure-preserving model-order reduction; passivity; spectral factorization; $\mathcal{H}_2$-optimal}

\begin{document}

\begin{abstract}
	We present a novel model-order reduction (MOR) method for linear time-invariant systems that preserves passivity and is thus suited for structure-preserving MOR for port-Hamiltonian (pH) systems. Our algorithm exploits the well-known spectral factorization of the Popov function by a solution of the Kalman-Yakubovich-Popov (KYP) inequality. It performs MOR directly on the spectral factor inheriting the original system's sparsity enabling MOR in a large-scale context. Our analysis reveals that the spectral factorization corresponding to the minimal solution of an associated algebraic Riccati equation is preferable from a model reduction perspective and benefits pH-preserving MOR methods such as a modified version of the iterative rational Krylov algorithm (IRKA). Numerical examples demonstrate that our approach can produce high-fidelity reduced-order models close to (unstructured) $\mathcal{H}_2$-optimal reduced-order models.
\end{abstract}

\maketitle
{\footnotesize \textsc{Keywords:} passivity, port-Hamiltonian systems, structure-preserving model-order reduction, spectral factorization, $\mathcal{H}_2$-optimal}

{\footnotesize \textsc{AMS subject classification:} 30E05,37M99,65P99,93A30,93A15, 93B99}
%
%
\section{Introduction}
We study structure-preserving model-order reduction methods for linear time-invariant (LTI) systems in standard state-space form presented as 
\begin{equation}
	\label{eqn:FOM}
	\Sigma = \left\{~\begin{aligned}
		\dot{x}(t) &= Ax(t) + Bu(t),\qquad x(0) = 0,\\
		y(t) &= Cx(t) + Du(t),
	\end{aligned}\right.
\end{equation}
where $u\colon\mathbb{R}\to\mathbb{R}^m$, $x\colon\mathbb{R}\to\mathbb{R}^n$, $y\colon\mathbb{R}\to\mathbb{R}^m$ are the \emph{input}, \emph{state}, and \emph{output} of the system. For convenience, we use the short notations $\Sigma = (A,B,C,D)$ and $G(s) = C(sI_n-A)^{-1}B + D$ to refer to the system~\eqref{eqn:FOM}. The matrix-valued function $G\colon\mathbb{C}\setminus\sigma(A)$ defined as
\begin{equation}
	G(s) \vcentcolon= C(sI_n-A)^{-1}B + D
\end{equation}  
is called the \emph{transfer function} of~\eqref{eqn:FOM}, which can be obtained by applying the Laplace transformation to~\eqref{eqn:FOM} and solving for the Laplace transformed state variable. Hereby, $\sigma(A)$ denotes the spectrum of~$A$, i.e., $\sigma(A) = \{s\in\mathbb{C} \mid \rank(sI_n-A)<n\}$.

The structure that we are interested in is that of \emph{passivity},  which implies that the system under investigation has a \emph{port-Hamiltonian} (pH) representation \cite{BeaMV19}. One of the many advantages of pH systems is that this model paradigm offers a systematic approach for the interaction of (physical) systems with each other and the environment via interconnection structure. Besides, the inherent structure of pH systems is amendable to structure-preserving approximation \cite{Egg19}. Consequently, the pH paradigm is particularly useful in future high-tech initiatives in systems engineering, such as a digital twin, where models are coupled across different scales and physical systems. For further details on pH systems we refer to \cite{DuiMSB09,JacZ12,VanJ14}, and \cite{BeaMXZ18}.

For many practically relevant examples, the dimension $n$ of \eqref{eqn:FOM} is too large to ensure an efficient simulation, control, or analysis of the system. This is all the more the case if the system results from a spatial semi discretization of a partial differential equation. The research field of \emph{model-order reduction (MOR)} aims to construct low-dimensional surrogate models that faithfully retain the original dynamics. In the particular case \eqref{eqn:FOM}, the precise goal would consist in finding a model, called reduced-order model (ROM), of the form
\begin{equation}
	\label{eqn:ROM}
	\tilde{\Sigma} = \left\{~\begin{aligned}
		\dot{\tilde{x}}(t) &= \tilde{A}\tilde{x}(t) + \tilde{B}u(t),\qquad \tilde{x}(0) = 0,\\
		\tilde{y}(t) &= \tilde{C}\tilde{x}(t) + \tilde{D}u(t),
	\end{aligned}\right.
\end{equation}
where $\tilde{x}\colon\mathbb{R}\to\mathbb{R}^r$ and $\tilde{y}\colon\mathbb{R}\to\mathbb{R}^m$ are the state and output of the \emph{reduced} system. The design of \eqref{eqn:ROM} typically comes with the concurrent goals of finding $r\ll n$ while at the same time guaranteeing that $y \approx \tilde{y}$. The latter approximation typically being formalized by choosing a specific system norm such as the $\mathcal{H}_2$-norm or the $\mathcal{H}_\infty$-norm, respectively. 

While a pH representation of a system is advantageous for coupling and interconnection of individual systems, it also comes with the additional challenge of preserving the structure within the reduction step. In particular, a classical \emph{Petrov-Galerkin} projection framework will generally destroy the pH structure. In view of this fact, several modifications of existing reduction techniques have been proposed. Model reduction for pH systems is considered from a balanced truncation point of view in, e.g., \cite{Pol10,PolR12,Wuetal18,BreMS20}. Interpolatory model reduction for pH systems is studied in \cite{PolV10,Woletal10,PolV11,GugPBS12,IonA13,Eggetal18}. Model reduction based on Riemannian and direct parameter optimization has been discussed in \cite{Sat18,MosL20,SchV20}. Since a state-space realization may not be available, recent works also focus on data-driven approaches that only rely on accessible (frequency domain) quantities \cite{BenGV20}. 

Additionally, there exists a rather extensive literature on passivity and positive realness preserving model reduction methods. Let us exemplarily mention generalized balancing based techniques \cite{DesP84,HarJS84,GugA04,ReiS10,GuiO13} as well as interpolatory methods relying on spectral zeros \cite{Ant05a,Sor05,MayA07}. Interestingly enough, while the aforementioned link between pH systems and passivity is well-known from a control-theoretic point of view, it appears that it has not been explicitly exploited in the model reduction context. One of our contributions is to partially close this gap by presenting a novel method with a clear system-theoretic understanding that can produce accurate low-dimensional surrogates, and that can compete with brute-force optimization of the system parameters \cite{SchV20}. Our main results are:
\begin{enumerate}
	\item  \Cref{thm:H2errorSpectralFactor}, which is based on a well-known spectral factorization of the Popov function, establishes a connection between the classical $\mathcal{H}_2$ model reduction error and the $\mathcal{H}_2$ reduction error of the associated spectral factors. 
	\item \Cref{alg:MORspectralFactor}, where we introduce a novel passivity preserving model reduction technique that enforces the ROM to satisfy a positive real Lur'e equation. Our technique builds upon a ROM for the spectral factor, which can be constructed with a MOR method of choice, provided that the ROM for the spectral factor is asymptotically stable.  In combination with \Cref{thm:H2errorSpectralFactor} this motivates to use IRKA~\cite{BeaG17} to construct ROM for the spectral factor with minimal $\mathcal{H}_2$ error. 
	\item Since \Cref{alg:MORspectralFactor} depends on a (particular) solution to the KYP inequality, in \Cref{thm:HankelSingularValuesSpectralFactor} we show that the minimal solution produces spectral factors with smallest Hankel singular values, deepening similar observations from \cite{BreMS20}.  This not only provides theoretical insight into our method, but also has an impact on other pH-preserving MOR methods such as pH-IRKA \cite{GugPBS12}, as is demonstrated in the numerical examples.
\end{enumerate}
    The precise structure is now as follows. In the subsequent section, we recall several classical results about passivity, positive realness, and port-Hamiltonian systems. Moreover, we state several model reduction methods for structured and unstructured systems relevant to the novel approach that we introduce and analyze in \cref{sec:MORspectralFactorization}.  \Cref{sec:examples} provides a detailed numerical study of our new method and compares its performance with other state-of-the-art reduction techniques. For this purpose, we show results for a mass-springer-damper system from \cite{BeaG17} and for a recently suggested pH formulation modeling poroelasticity \cite{AltMU20c}.
    
\paragraph*{Notation.} By $\mathbb{R}$ and $\mathbb{C}$ we denote the set of real and complex numbers.  Furthermore, we use the symbols
	\begin{gather*}
		\overline{\mathbb{R}}_{+} \vcentcolon= \{x\in \mathbb{R} \mid x\geq 0\},\quad
		\mathbb{C}_+ \vcentcolon= \{z\in \mathbb{C} \mid \Real(z) > 0\},\quad 
		\mathbb{C}_- \vcentcolon= \{z\in \mathbb{C} \mid \Real(z) < 0\},
	\end{gather*}	
	to denote the non-negative real numbers, and the open right and left half complex plane. 
For a matrix $A\in \mathbb R^{n\times n}$ its transpose, symmetric, and unsymmetric part is given by $A^\top$, $\Sym(A)$ and $\Skew(A)$, respectively.  The identity matrix of dimension $n$ is denoted by $I_n$. For symmetric matrices $A,B\in \mathbb R^{n\times n}$, we use $A\ge B$ if $A-B$ is positive semidefinite. For a matrix $V\in \mathbb R^{n\times m}$, we denote its associated column space by $\Ran(V)$. The Frobenius norm of a matrix $A\in \mathbb R^{m\times n}$ is denoted by $\|A\|_{\mathrm{F}}$. For a dynamical system and its transfer function $G$, let us recall the classical spaces
\begin{align*}
	\mathcal{H}_2(\mathbb{C}_+) &\vcentcolon= \left\{ G\colon\mathbb{C}_+ \to\mathbb{C}^{m\times m} ~\left|~
	G \text{ is analytic and } \|G\|_{\mathcal{H}_2(\mathbb{C}_+)} < \infty
	\right.\right\},\\
	\mathcal{H}_\infty(\mathbb{C}_+) &\vcentcolon= \left\{ G\colon\mathbb{C}_+ \to\mathbb{C}^{m\times m} ~\left|~
		G \text{ is analytic and }\|G\|_{\mathcal{H}_\infty(\mathbb{C}_+)} < \infty
	\right.\right\},
\end{align*}
with
\begin{align*}
	\|G\|_{\mathcal{H}_2(\mathbb{C}_+)} &\vcentcolon= \left(\sup_{\sigma >0 } 
\int_{-\infty}^{\infty} \| G(\sigma + \imath \omega ) \| _F ^2 \ \dd
\omega \right)^{\frac{1}{2}},\quad 
	\|G\| _{\mathcal{H}_\infty(\mathbb{C}_+)} \vcentcolon=\displaystyle \sup_{z \in \mathbb C^+} \| G (z) \| _2.
\end{align*}
Similarly, we will consider $\mathcal{H}_2(\mathbb C_-),\mathcal{H}_\infty(\mathbb C_-), \mathcal{L}_2(i\mathbb R)$ and $\mathcal{L}_\infty(i\mathbb R)$.

\section{Preliminaries}

\subsection{Passive, positive real, and port-Hamiltonian systems}
In this section we recall dissipation theory for dynamical systems, which is used later on in the construction of the ROM. Throughout the text we assume that~\eqref{eqn:FOM} is \emph{minimal}, i.e.,  for all $s\in\C$ the conditions
\begin{align*}
	\rank\begin{bmatrix}
		sI_n-A, & B
	\end{bmatrix} = n =
	\rank\begin{bmatrix}
		sI_n-A^\top, & C^\top
	\end{bmatrix}
\end{align*}
are satisfied. The \emph{Popov function} for~\eqref{eqn:FOM} is defined as
\begin{equation}
	\label{eqn:popovFunction}
	\begin{gathered}
	\Phi\colon\C\setminus \big(\sigma(A)\cup \sigma(-A)\big) \to\C^{m\times m},\quad 
	s\mapsto G(s) + G(-s)^\top.
	\end{gathered}
\end{equation}

\begin{definition}
	\label{def:posRealPassPH}
	We consider system~\eqref{eqn:FOM}.
	\begin{enumerate}
		\item System~\eqref{eqn:FOM} is called \emph{positive real}, if the Popov function~\eqref{eqn:popovFunction} is positive semidefinite on the imaginary axis, i.e.
			\begin{equation}
				\label{eqn:positiveReal}
				\Phi(\imath\omega)\geq 0 \qquad\text{for all $\omega\in\R$}.
			\end{equation}
			It is called \emph{strictly positive real} if the inequality in~\eqref{eqn:positiveReal} is strict.
		\item System~\eqref{eqn:FOM} is called \emph{passive}, if there exists a state-dependent \emph{storage function} $\mathcal{H}\colon \R^n\to \overline{\mathbb{R}}_+$ satisfying for any $t_1\geq t_0$ the dissipation inequality
			\begin{equation}
				\label{eqn:dissipationInequality}
				\mathcal{H}(x(t_1)) - \mathcal{H}(x(t_0)) \leq \int_{t_0}^{t_1} y(\tau)^\top u(\tau)\mathrm{d}\tau
			\end{equation}
			for arbitrary trajectories $u,x,y$ satisfying~\eqref{eqn:FOM}.
		\item System~\eqref{eqn:FOM} is called \emph{port-Hamiltonian} (pH), if there exist a symmetric positive definite matrix $Q = Q^\top \in \R^{n\times n}$ and matrix decompositions $A = (J-R)Q$, $B = G-P$, $C = (G+P)^\top Q$, $D = S+N$, satisfying
			\begin{equation}
				\label{eqn:pHMatrixProperties}
				\begin{bmatrix}
					-J & -G\\
					G^\top & N
				\end{bmatrix}^\top = -\begin{bmatrix}
					-J & -G\\
					G^\top & N
				\end{bmatrix} \qquad\text{and}\qquad
				\begin{bmatrix}
					R & P\\
					P^\top & S
				\end{bmatrix}^\top = \begin{bmatrix}
					R & P\\
					P^\top & S
				\end{bmatrix} \geq 0.
			\end{equation}			
			In this case, we call
			\begin{equation}
				\label{eqn:pH}
				\left\{~\begin{aligned}
		\dot{x}(t) &= (J-R)Qx(t) + (G-P)u(t),\\
		y(t) &= (G+P)^\top Qx(t) + (S+N)u(t)
	\end{aligned}\right.
			\end{equation}
			a \emph{pH representation} of~\eqref{eqn:FOM}.
	\end{enumerate}
\end{definition}

\begin{remark}
	\label{rem:pHdescriptor}
	In practice, a pH representation is often directly available after modeling, albeit sometimes in generalized state-space form (also referred to as co-energy representation) 
	\begin{equation}
		\label{eqn:pHdescriptor}
				\left\{~\begin{aligned}
		E\dot{x}(t) &= (J-R)x(t) + (G-P)u(t),\\
		y(t) &= (G+P)^\top x(t) + (S+N)u(t)
	\end{aligned}\right.
	\end{equation}
	with symmetric positive definite $E$. The remaining matrices have to satisfy the same properties as in the standard state-space case given in~\eqref{eqn:pHMatrixProperties}.   Although this representation is known to be favorable for numerical approximation \cite{Egg19}, see also \cite{CarML20}, and easily extendable to descriptor systems \cite{BeaMXZ18}, we work with the representation~\eqref{eqn:pH} and consider extensions to~\eqref{eqn:pHdescriptor} in future work.
\end{remark}

With the matrix function $\mathcal{W}\colon \R^{n\times n} \to \R^{(n+m)\times (n+m)}$ defined via
\begin{equation}
	\label{eqn:KYPmatrix}
	\mathcal{W}(X) \vcentcolon= \begin{bmatrix}
		-A^\top X - XA & C^\top - XB\\
		C - B^\top X & D + D^\top
	\end{bmatrix},
\end{equation}
the Popov function can be factorized as
\begin{equation}
	\label{eqn:PopovFactorization}
	\Phi(s) = \begin{bmatrix}
		(-sI_n-A)^{-1}B\\
		I_m
	\end{bmatrix}^\top \mathcal{W}(X) \begin{bmatrix}
		(sI_n-A)^{-1}B\\
		I_m
	\end{bmatrix}.
\end{equation}

With these preparations, we have the following equivalence, see for instance~\cite{BeaMXZ18}.

\begin{theorem}
	\label{thm:equivalence}
	Assume that~\eqref{eqn:FOM} is minimal and stable. Then the following are equivalent.
	\begin{enumerate}
		\item The system~\eqref{eqn:FOM} is positive real.
		\item The system~\eqref{eqn:FOM} is passive.
		\item The system~\eqref{eqn:FOM} is port-Hamiltonian.
		\item There exists a symmetric positive definite matrix $X\in \R^{n\times n}$ satisfying the KYP inequality
			\begin{equation}
				\label{eqn:KYPinequality}
				\mathcal{W}(X) \geq 0.
			\end{equation}
	\end{enumerate}
\end{theorem}

\begin{remark}
	\label{rem:pHrepresentation}
	For a passive system, we immediately observe that 
	whenever we have a positive definite solution $X = X^\top > 0$ of the KYP inequality~\eqref{eqn:KYPinequality}, then a port-Hamiltonian representation may be obtained via \begin{gather*}
		Q \vcentcolon= X, \qquad  J \vcentcolon= \tfrac{1}{2}\left(AX^{-1} - X^{-1}A^\top\right),  \qquad R \vcentcolon= -\tfrac{1}{2}\left(AX^{-1} + X^{-1}A^\top\right),\\
		G \vcentcolon= \tfrac{1}{2}(X^{-1}C^\top + B), \qquad P \vcentcolon= \tfrac{1}{2}(X^{-1}C^\top - B), \qquad S\vcentcolon= \Sym(D), \qquad N\vcentcolon= \Skew(D).
	\end{gather*}
	For the details we refer to \cite{BeaMV19}. Let us emphasize that different decompositions do not correspond to different state-space realizations as one would obtain by a change of coordinates. Consequently, this allows to keep the matrices $A = (J-R)Q$, $B = (G-P)$, $C = (G+P)^\top Q$, and $D = S+N$ unchanged while changing the system Hamiltonian from $Q$ to $X$. Later on, this property will be utilized to construct a system Hamiltonian which is particularly well suited for model reduction purposes.
\end{remark}

\subsection{Solutions of the KYP inequality}
\label{subsec:KYP}
Since our MOR algorithm relies on a solution of the KYP inequality~\eqref{eqn:KYPinequality}, we will briefly discuss related theoretical results and numerical methods. If the system $(A,B,C,D)$ is pH (cf.~Definition~\ref{def:posRealPassPH}), then we immediately observe that
\begin{align*}
	\mathcal{W}(Q) &= \begin{bmatrix}
		-\left((J-R)Q\right)^\top Q - Q(J-R)Q & Q(G+P) - Q(G-P)\\
		(G+P)^\top Q - (G-P)^\top Q & S+N + S^\top +  N^\top
	\end{bmatrix}\\
	&= \begin{bmatrix}
		-Q(J-R)^\top Q - Q(J-R)Q & 2QP\\
		2P^\top Q & 2S
	\end{bmatrix} = \begin{bmatrix}
		2QRQ & 2QP\\
		2P^\top Q & 2S
	\end{bmatrix}\\
	&= 2\begin{bmatrix}
		Q & 0\\
		0 & I_m
	\end{bmatrix}\begin{bmatrix}
		R & P\\
		P^\top & S
	\end{bmatrix}\begin{bmatrix}
		Q & 0\\
		0 & I_m
	\end{bmatrix} \geq 0,
\end{align*}
i.e., $Q$ solves the KYP inequality~\eqref{eqn:KYPinequality}. However, in general we cannot expect that $\mathcal{W}(Q)$ is of minimal rank amongst all solutions of the KYP inequality~\eqref{eqn:KYPinequality}. If we are interested in solutions $X$ such that $\mathcal{W}(X)$ is of minimal rank, then we can consider the closely related \emph{Lur'e equations}
\begin{subequations}
	\label{eqn:LureEquations}
	\begin{align}
		-A^\top X - XA &= L^\top L,\\
		XB - C^\top &= L^\top M,\\
		D + D^\top &= M^\top M,
	\end{align}
\end{subequations}
which have to be solved for the triple $(X, L, M)\in\R^{n\times n}\times \R^{k\times n} \times \R^{k\times m}$ with symmetric positive definite $X$ and $p = \rank\left[\begin{smallmatrix}
	L & M
\end{smallmatrix}\right]$ as small as possible. Clearly, for any solution $X$ of the KYP-inequality, there exists $L$ and $M$ (not necessarily of minimal rank) satisfying the Lur'e equations~\eqref{eqn:LureEquations}, and, vice versa, if $(X,L,M)$ solves the Lur'e equations~\eqref{eqn:LureEquations}, then $X$ is a solution of the KYP inequality~\eqref{eqn:KYPinequality}. In this case, we have the Cholesky-like factorization
\begin{equation}
	\label{eqn:KYPchol}
	\mathcal{W}(X) = \begin{bmatrix}
		L^\top\\M^\top
	\end{bmatrix}\begin{bmatrix}
		L & M
	\end{bmatrix} = \begin{bmatrix}
		L^\top L & L^\top M\\
		M^\top L & M^\top M
	\end{bmatrix}.
\end{equation}
For a general analysis of Lur'e equations and their relation to even matrix pencils we refer to~\cite{Rei11} and the references therein.

\subsubsection{The regular case} 

If $D+D^\top$ is nonsingular, then one can eliminate the unknowns $L$ and $M$, and use the Schur complement to derive the algebraic Riccati equation~(ARE)
\begin{equation}
	\label{eqn:ARE}
	-A^\top X - XA - (C^\top -XB)(D + D^\top)^{-1}(C-B^\top X) = 0.
\end{equation}
Indeed, note that we may set $L=(D+D^\top)^{-\frac{1}{2}}(C-B^\top X), M=(D+D^\top)^{\frac{1}{2}} $ and obtain $L\in \mathbb R^{m\times n}$, where $m$ is the number of inputs/outputs. 

It is well-known, see for instance \cite{Wil71}, that the symmetric solutions of~\eqref{eqn:ARE} are bounded, i.e., there exist symmetric matrices $\Xmin$ and $\Xmax$ solving~\eqref{eqn:ARE} such that for any symmetric solution $X$ of the ARE~\eqref{eqn:ARE}, the inequalities $\Xmin \leq X \leq \Xmax$ are satisfied. If we additionally assume (as we will do throughout the manuscript) that the system $\Sigma = (A,B,C,D)$ is minimal and passive, then these solutions are positive definite, i.e.,
\begin{displaymath}
	0 < \Xmin \leq X \leq \Xmax.
\end{displaymath}

\subsubsection{The singular case}

In case that $D+D^\top$ is singular, one can replace $D$ with the perturbation $D_{\varepsilon} \vcentcolon= D+\tfrac{\varepsilon}{2}I_m$ for some $\varepsilon > 0$ and solve the perturbed ARE
\begin{equation*}
	-A^\top X_{\varepsilon} - X_{\varepsilon} A - (C^\top -XB)(D + D^T + \varepsilon I_m)^{-1}(C-B^\top C) = 0.
\end{equation*}
Indeed, the minimal and maximal solutions $X_{\varepsilon,\min}$ and $X_{\varepsilon,\max}$ converge for $\varepsilon\to 0$, and the limiting matrices satisfy the Lur'e equations and thus the KYP inequality. For details, we refer to~\cite[Thm.~2]{Wil72b}. From a computational perspective, one may argue, cf.~\cite{PolR12}, that the numerical sensitivity of the Riccati equation~\eqref{eqn:ARE} increases with $\varepsilon$ tending to zero and that no convergence rates and no estimates for the error $\|X-X_{\varepsilon}\|$ are available. Instead, one may, for instance, solve the Lur'e equation by deflating its singular part~\cite{PolR12} or use an ADI iteration \cite{MasOR17}.  For an overview of existing methods to compute solutions of the Lur'e equations, we refer to~\cite{PolR12} and \cite{MasOR17} and the references therein. To the best of our knowledge, the existing methods first compute an (approximate) solution $X$ and then, in a post-processing step, compute the factors $L$ and $M$. A method suitable for the large-scale context that solely focuses on the computation of $L$ and $M$ seems not available. 

\subsection{Positive real balanced truncation}
The task of standard projection-based model reduction methods is to determine matrices $V,W\in\R^{n\times r}$ with $W^\top V = I_r$. The ROM~\eqref{eqn:ROM} associated with these matrices is given by
\begin{align*}
	\tilde{A} &\vcentcolon= W^\top AV, &
	\tilde{B} &\vcentcolon= W^\top B, &
	\tilde{C} &\vcentcolon= CV, &
	\tilde{D} &\vcentcolon= D
\end{align*}
with transfer function $\tilde{G}(s) = \tilde{C}(sI_r-\tilde{A})^{-1}\tilde{B}+\tilde{D}$.  The choice of the matrices $V,W$ used for the Petrov-Galerkin projection determines the approximation quality and additional properties of the ROM. Positive real balanced truncation, see \cite{DesP84,GugA04,HarJS84} for details, determines $V$ and $W$ by first, balancing the minimal solutions of the Lur'e equation \eqref{eqn:LureEquations} and its dual version 
\begin{subequations}
	\label{eqn:LureEquations_dual}
	\begin{align}
		-A Y - YA^\top &= \widehat{L} \widehat{L}^\top,\\
		YC^\top - B &= \widehat{L} \widehat{M}^\top,\\
		D + D^\top &= \widehat{M}\widehat{M}^\top .
	\end{align}
\end{subequations}
Note that if $(X,L,M)$ with $X>0$ solve \eqref{eqn:LureEquations}, then $(X^{-1},X^{-1}L^\top,M^\top)$ solve \eqref{eqn:LureEquations_dual} and vice versa. For strictly positive real systems it holds that $D+D^\top>0$ and instead of \eqref{eqn:LureEquations} and \eqref{eqn:LureEquations_dual} one may solve the positive real algebraic Riccati equations (AREs)
\begin{subequations}
	\label{eqn:positiveRealARE}
	\begin{align}
		A^\top X + XA + (C^\top-XB)(D+D^\top)^{-1}(C-B^\top X) &= 0, \\
		A Y + Y^\top A + (B-YC^\top)(D+D^\top)^{-1} (B^\top -CY) &= 0.
	\end{align}
\end{subequations}
A reduced system is subsequently obtained by truncation w.r.t.~the balanced coordinates. For strictly positive real systems with $X=Y=\mathrm{diag}(\sigma_1,\dots,\sigma_n)$ with $\sigma_i>\sigma_{i+1}$ for $i=1,\dots,n-1,$ this procedure is known to preserve asymptotic stability, passivity, and minimality, see \cite{HarJS84}. Additionally, different (relative) $\mathcal{H}_\infty$ error bounds have been derived, see, e.g., \cite{GugA04} for further details. Let us emphasize that in \cite{GuiO13} an error bound with respect to the \emph{gap metric} has been proven. Besides being structurally similar to the classical a priori $\mathcal{H}_\infty$ error bound, this additionally provides an interpretation of the closed loop behavior, which is independent of any input-output decomposition. The details for computing a positive real balanced reduced system are given in \Cref{alg:prbt}.
\begin{algorithm}[ht]
	\caption{Positive real balanced truncation}
	\label{alg:prbt}
	\begin{algorithmic}[1]
		\Statex \textbf{Input:} passive system $(A,B,C,D)$, reduced order  $r\in\mathbb{N}$
		\Statex \textbf{Output:} reduced passive system $(\tilde{A},\tilde{B},\tilde{C},\tilde{D})$
		\State Compute the minimal solutions
			\begin{displaymath}
				X_{\min}=L_X^\top L_X,\qquad Y_{\min}=L_Y^\top L_Y
			\end{displaymath}
			to the positive real Lur'e equations~\eqref{eqn:LureEquations} and~\eqref{eqn:LureEquations_dual}, respectively.
		\State Compute the singular value decomposition
			\begin{displaymath}
				\begin{bmatrix} U_1 & U_2 \end{bmatrix}\begin{bmatrix}\Sigma_1& 0 \\ 0 & \Sigma_2 \end{bmatrix} \begin{bmatrix}Z_1^\top \\ Z_2^\top \end{bmatrix}= L_Y L_X^\top.
				\end{displaymath}
		\State Define $V\vcentcolon=L_Y^\top U_1 \Sigma_1^{-\frac{1}{2}}$ and $W\vcentcolon=L_X^\top Z_1 \Sigma_1^{-\frac{1}{2}}$.
		\State Set $\tilde{A}\vcentcolon=W^\top AV, \tilde{B}\vcentcolon=W^\top B, \tilde{C}\vcentcolon=CV$, and $\tilde{D}\vcentcolon=D$.
	\end{algorithmic}
\end{algorithm}

\subsection{Structure-preserving model reduction via interpolation}
Interpolatory model reduction is a well-known technique, see~\cite{BeaG17} or \cite{AntBG20}, that constructs ROMs whose transfer function interpolates the transfer function of the original model at selected interpolation points.  In a projection framework, this can be achieved as follows;  \cite[Thm.~7.1]{BeaG17}.

\begin{theorem}[{Rational interpolation}]
	\label{thm:rationalInterp}
	\ \\
	 Consider the dynamical system~\eqref{eqn:FOM} with transfer function $G(s)$ and the ROM~\eqref{eqn:ROM} with transfer function $\tilde{G}(s)$ constructed via projection with the matrices $W,V\in\R^{n\times r}$. For interpolation points $\lambda,\mu\in\C$ assume that $\lambda I_n-A$ and $\mu I_n -A$ are nonsingular, i.e., $\lambda,\mu\not\in\sigma(A)$.  Let $\mathsf{r}\in\R^m$and $\ell\in\R^m$. 
	\begin{enumerate}
		\item\label{thm:rationalInterp:right} If $(\lambda I_n -A)^{-1}B\mathsf{r} \in \Ran(V)$, then $G(\lambda)\mathsf{r} = \tilde{G}(\lambda)\mathsf{r}$.
		\item\label{thm:rationalInterp:left} If $(\ell^\top C(\mu I_n-A)^{-1})^\top \in \Ran(W)$, then $\ell^\top G(\mu) = \ell^\top \tilde{G}(\mu)$.
		\item If the conditions in~(\ref{thm:rationalInterp:right}) and~(\ref{thm:rationalInterp:left}) are simultaneously satisfied with $\lambda = \mu$, then $\ell^\top G'(\lambda)\mathsf{r} = \ell^\top \tilde{G}'(\lambda)\mathsf{r}$.
	\end{enumerate}
\end{theorem}

While \Cref{thm:rationalInterp} details the construction of a ROM for given interpolation points $\lambda, \mu$ and tangent directions $\mathsf{r},\ell$, it does not provide a strategy to choose these quantities leading to a high-fidelity ROM. In the context of $\mathcal{H}_2$ optimal reduced-order models, the following theorem~\cite[Thm.~7.7]{BeaG17} provides an implicit definition.

\begin{theorem}[{$\mathcal{H}_2$-optimality conditions}]
	\label{thm:H2optimal}
	\ \\
	Let~\eqref{eqn:ROM} with semi-simple eigenvalues $\lambda_1,\ldots,\lambda_r$ of\ $\tilde{A}$ and transfer function $\tilde{G}(s) = \sum_{i=1}^r \tfrac{\ell_i \mathsf{r}_i^\top}{s-\lambda_i}$ be the best-approximation of~\eqref{eqn:FOM} with transfer function $G(s)$ of dimension $r$ with respect to the $\mathcal{H}_2$-norm.Then,
	\begin{subequations}
		\label{eqn:H2optimalityConditions}
	\begin{align}
		G(-\lambda_k)\mathsf{r}_k &= \tilde{G}(-\lambda_k)\mathsf{r}_k,\\
		\ell_k^\top G(-\lambda_k) &= \ell_k^\top \tilde{G}(-\lambda_k), \qquad\text{and} \label{h2_left_oc}\\
		\ell_k^\top G'(-\lambda_k)\mathsf{r}_k &= \ell_k^\top \tilde{G}'(-\lambda_k)\mathsf{r}_k
	\end{align} 
	\end{subequations}
	for $k=1,\ldots,r$.
\end{theorem}

\Cref{thm:H2optimal} generalized results known for the scalar case \cite{MeiL67} and motivated the construction of an algorithm that iteratively updates interpolation points and tangent directions until the necessary optimality conditions~\eqref{eqn:H2optimalityConditions} are satisfied by utilizing \Cref{thm:rationalInterp}. The iterative rational Krylov algorithm (IRKA) \cite{GugAB08}, is such an algorithm. The details are presented in \Cref{alg:IRKA}.

\begin{algorithm}[ht]
	\caption{IRKA (\!\!\cite{GugAB08})}
	\label{alg:IRKA}
	\begin{algorithmic}[1]
		\Statex \textbf{Input:} original system $(A,B,C)$, reduced-order $r\in\mathbb{N}$
		\Statex \textbf{Output:} reduced system $(\tilde{A},\tilde{B},\tilde{C})$ of order $r$
		\State Choose initial interpolation points $\{s_1,\ldots,s_r\}$ and tangent directions $\{\mathsf{r}_1,\ldots,\mathsf{r}_r\}$, $\{\ell_1,\ldots,\ell_r\}$. All sets closed under conjugation.
		\State Construct real matrices $V,W\in\R^{n\times r}$ satisfying $W^\top V=I$ and
			\begin{align*}
				(s_iI_n - A)^{-1}B\mathsf{r}_i \in \Ran(V),\\
                (s_iI_n - A^\top)^{-1}C^\top \ell_i \in \Ran(W).
			\end{align*}
		\Repeat 
			\State Compute $\tilde{A} \vcentcolon= W^\top AV$, $\tilde{B} \vcentcolon= W^\top B$, $\tilde{C}\vcentcolon=CV$.
			\State Compute a pole-residue expansion of 
			\begin{displaymath}
                \tilde{G}(s)=\tilde{C}(sI-\tilde{A})^{-1}\tilde{B}=\sum_{i=1}^r \frac{c_i b_i^\top}{s-\lambda_i}.
            \end{displaymath} 
			\State Set $s_i = -\lambda_i$, $\mathsf{r}_i = b_i$ and $\ell_i=c_i$, for $i=1,\ldots,r$.
			\State Construct real matrices $V,W\in\R^{n\times r}$ satisfying $W^\top V=I$ and
			\begin{align*}
				(s_iI_n - A)^{-1}B\mathsf{r}_i \in \Ran(V),\\
                (s_iI_n - A^\top)^{-1}C^\top \ell_i \in \Ran(W).
			\end{align*}
		\Until convergence
	\end{algorithmic}
\end{algorithm}

Since the ROM constructed via IRKA (cf.~\Cref{alg:IRKA}) is obtained via Petrov-Galerkin projection, it is a priori not clear that passivity or the pH structure is preserved, and in general, this will not be the case.  Instead, one can aim for a ROM that only satisfies a subset of the necessary optimality conditions~\eqref{eqn:H2optimalityConditions} and use the remaining degrees of freedom to enforce the pH structure. This can for instance be achieved by constructing $V$ similar as in \Cref{alg:IRKA} and choose $W\vcentcolon= QV(V^\top QV)^{-1}$. This particular choice then directly yields a pH realization. The corresponding modification of \Cref{alg:IRKA} is presented in \Cref{alg:pHIRKA}, originally introduced in \cite{GugPBS12}. Let us emphasize that the choice of $W$ depends on the particular pH representation; see \Cref{rem:pHrepresentation}. Besides the drawback of limiting the degrees of freedom of a Petrov-Galerkin projection, we also expect that the approximation quality of ROMs for different choices of $Q$ varies.

\begin{algorithm}[ht]
	\caption{IRKA-pH (\!\!\cite{GugPBS12})}
	\label{alg:pHIRKA}
	\begin{algorithmic}[1]
		\Statex \textbf{Input:} pH system with $A=(J-R)Q$, $B=G-P$, $C=(G+P)^\top Q$, and $D=S-N$, reduced-order $r\in\mathbb{N}$
		\Statex \textbf{Output:} reduced pH system with $\tilde{A} = (\tilde{J}-\tilde{R})\tilde{Q}$, $\tilde{B}=\tilde{G}-\tilde{P}$, $\tilde{C} = (\tilde{G}+\tilde{P})^\top \tilde{Q}$, $\tilde{D} = S-N$
		\State Choose initial interpolation points $\{s_1,\ldots,s_r\}$ and tangent directions $\{b_1,\ldots,b_r\}$. Both sets closed under conjugation.
		\State Construct a real matrix $V\in\R^{n\times r}$ satisfying
			\begin{displaymath}
				(s_iI_n - (J-R)Q)^{-1}Bb_i \in \Ran(V).
			\end{displaymath}
		\State Calculate $W\vcentcolon= QV(V^\top QV)^{-1}$.
		\Repeat 
			\State Compute $\tilde{J} \vcentcolon= W^\top JW$, $\tilde{R} \vcentcolon= W^\top RW$, $\tilde{Q} \vcentcolon= V^\top QV$, $\tilde{G} \vcentcolon= W^\top G$, $\tilde{P}
			\vcentcolon= W^\top P$.
			\State For $\tilde{A} \vcentcolon= (\tilde{J}-\tilde{R})\tilde{Q}$ compute the $r$ eigenvalues $\lambda_i$ and associated left eigenvectors $y_i$. 
			\State Set $s_i = -\lambda_i$ and $b_i^\top = y_i^\top (G+P)$ for $i=1,\ldots,r$.
			\State Construct a real matrix $V\in\R^{n\times r}$ satisfying
			\begin{displaymath}
				(s_iI_n - (J-R)Q)^{-1}Bb_i \in \Ran(V).
			\end{displaymath}
			\State Calculate $W\vcentcolon= QV(V^\top QV)^{-1}$.
		\Until convergence
	\end{algorithmic}
\end{algorithm}

\begin{remark}
	Interpolatory methods can also be applied solely from data, for instance, via the Loewner framework \cite{MayA07}.  Suppose the frequency points are chosen as the so-called \emph{spectral zeros}. In that case, passivity is retained with the Loewner framework \cite[Sec.~8.2.4]{AntLI17}. Since the spectral zeros are typically not available in the data-driven regime, \cite{BenGV20} propose to construct a realization of the full-order model, which in turn can be used to infer the spectral zeros.  Further methods to construct a pH realization from data are considered, for instance, in~\cite{CheMH19}.
\end{remark}

\section{Passivity-preserving MOR via spectral factorization}
\label{sec:MORspectralFactorization}

Suppose that~\eqref{eqn:FOM} is passive. Then, by virtue of \Cref{thm:equivalence}, there exists a symmetric positive semidefinite matrix $X\in\R^{n\times n}$ satisfying the KYP inequality~\eqref{eqn:KYPinequality}.  Since $\mathcal{W}(X)$ is positive semidefinite, we can factorize 
\begin{displaymath}
	\mathcal{W}(X) = \begin{bmatrix}
		L^\top\\M^\top
	\end{bmatrix}\begin{bmatrix}
		L & M
	\end{bmatrix} = \begin{bmatrix}
		L^\top L & L^\top M\\
		M^\top L & M^\top M
	\end{bmatrix}
\end{displaymath}
with $L\in\R^{k\times n}$ and $M\in\R^{k\times m},k\le n,$ similarly as in~\eqref{eqn:KYPchol}. Define the auxiliary system,
\begin{equation}
	\label{eqn:auxiliarySystem}
	\Sigma_H = \left\{~\begin{aligned}
		\dot{x}(t) &= Ax(t) + Bu(t),\qquad x(0) = 0,\\
		y(t) &= Lx(t) + Mu(t),
	\end{aligned}\right.
\end{equation}
which we refer to as a \emph{spectral factor} of $\Sigma$,  by replacing the matrices in the output equation of~\eqref{eqn:FOM} with the Cholesky factors $L$ and $M$. For a detailed treatise of spectral factorizations in the context of control systems, we refer to \cite[Section 13.4]{ZhoDG96}. The spectral factor's transfer function is given by $H(s) \vcentcolon= M + L(sI_n-A)^{-1}B$. Using the factorization~\eqref{eqn:PopovFactorization} we thus obtain the spectral factorization
\begin{align*}
	 H^\top(-s)H(s) &= \begin{bmatrix}
		(-sI-A)^{-1}B\\
		I_m
	\end{bmatrix}^\top \mathcal{W}(X) \begin{bmatrix}
		(sI-A)^{-1}B\\
		I_m
	\end{bmatrix}= \Phi(s)
\end{align*}
of the Popov function.  The following result details that the $\mathcal{H}_2$ difference between two passive systems can be bounded by the $\mathcal{H}_2$ difference of the associated spectral factors.

\begin{theorem}
	\label{thm:H2errorSpectralFactor}
    Consider passive,  minimal, and asymptotically stable systems $\Sigma$ and $\tilde{\Sigma}$ with transfer functions $G$ and $\tilde{G}$. Let $H$ and $\tilde{H}$ denote transfer functions of associated spectral factors as in~\eqref{eqn:auxiliarySystem}. If $G-\tilde{G},H-\tilde{H}\in \mathcal{H}_2$, then 
    \begin{align}
    \label{eqn:h2_g_vs_h}
    \|G-\tilde{G}\|_{\mathcal{H}_2(\mathbb C_+)} \le c(H,\tilde{H}) \| H(\cdot)-\tilde{H}(\cdot)\| _{\mathcal{H}_2(\mathbb C_+)}
    \end{align}
    with $c(H,\tilde{H}) \vcentcolon= \frac{1}{\sqrt{2}} \big(\|H^\top(-\cdot)\|_{L^{\infty}(\imath \mathbb R)}+\|\tilde{H}(\cdot)\|_{L^{\infty}(\imath \mathbb R)}\big)$.
\end{theorem}

\begin{proof}
    Let us first note that since $G-\tilde{G}\in \mathcal{H}_2$, from \cite[Lem.~A.6.18]{CurZ95} we obtain
    \begin{align*}
        \|G(\cdot)-\tilde{G}(\cdot)\|_{\mathcal{H}_2(\mathbb C_+)}^2
        &= \sup\limits_{\sigma> 0} \int_{-\infty}^{\infty} \| G(\sigma+\imath \omega) - \tilde{G}(\sigma+\imath \omega)\|_{\mathrm{F}}^2\; \mathrm{d}\omega \\
        &= \int_{-\infty}^{\infty} \| G(\imath \omega) - \tilde{G}(\imath \omega) \| _{\mathrm{F}}^2 \;\mathrm{d} \omega = \int_{-\infty}^{\infty} \| G^\top(-\imath \omega) - \tilde{G}^\top (-\imath \omega) \| _{\mathrm{F}}^2 \;\mathrm{d} \omega \\
        &=\sup\limits_{\sigma> 0} \int_{-\infty}^{\infty} \| G^\top(-\sigma-\imath \omega) - \tilde{G}^\top(-\sigma-\imath \omega)\|_{\mathrm{F}}^2\; \mathrm{d}\omega\\
        &= \|G^\top(-\cdot)-\tilde{G}^\top(-\cdot)\|_{\mathcal{H}_2(\mathbb C_-)}^2.
    \end{align*}
    Similarly, it follows that
    \begin{displaymath}
    		\|H(\cdot)-\tilde{H}(\cdot)\|_{\mathcal{H}_2(\mathbb C_+)} = \|H^\top(-\cdot)-\tilde{H}^\top(-\cdot)\|_{\mathcal{H}_2(\mathbb C_-)}.
    	\end{displaymath}
    	From the well known orthogonal decomposition $L^2(\imath \mathbb R)=\mathcal{H}_2(\mathbb C_+) \oplus \mathcal{H}_2(\mathbb C_-)$, see, e.g., \cite[Thm.~A.6.22]{CurZ95}, we conclude $\| \Phi(\cdot) - \tilde{\Phi}(\cdot) \| _{L^2(\imath \mathbb R)} = \sqrt{2}\| G(\cdot)-\tilde{G}(\cdot)\|_{\mathcal{H}_2(\mathbb C_+)}$. We further obtain that 
    \begin{align*}
     \Phi(\cdot)-\tilde{\Phi}(\cdot)     &= H^\top(-\cdot)H(\cdot)-\tilde{H}^\top(-\cdot)\tilde{H}(-\cdot) \\
     &= H^\top(-\cdot)(H(\cdot)-\tilde{H}(-\cdot)) + (H^\top(-\cdot)-\tilde{H}^\top(-\cdot))\tilde{H}(\cdot).
    \end{align*}
    From the asymptotic stability of $A$ and $\tilde{A}$, we conclude $H^\top(-\cdot),\tilde{H}(\cdot) \in L^\infty(\imath \mathbb R)$, which, together with the above equality, yields 
    \begin{align*}
     \| \Phi(\cdot)-\tilde{\Phi}(\cdot)\|_{L^2(\imath \mathbb R)} \le \sqrt{2}c(H,\tilde{H})  \| H(\cdot)-\tilde{H}(\cdot)\| _{\mathcal{H}_2(\mathbb{C}_+)}.
    \end{align*}
\end{proof}

Although~\eqref{eqn:h2_g_vs_h} is not an a-priori error bound,  \Cref{thm:H2errorSpectralFactor} immediately suggests to construct a reduced-order model for the spectral factor $\Sigma_H$, and then construct a passive ROM $\tilde{\Sigma}$ by reversing the construction of $\Sigma_H$. The details are presented in \Cref{alg:MORspectralFactor}.  
\begin{algorithm}[ht]
	\caption{Passivity preserving MOR via spectral factors}
	\label{alg:MORspectralFactor}
	\begin{algorithmic}[1]
		\Statex \textbf{Input:} passive system $(A,B,C,D)$, reduced-order $r\in\mathbb{N}$
		\Statex \textbf{Output:} passive ROM $(\tilde{A},\tilde{B},\tilde{C},\tilde{D})$ of order $r$
		\State\label{alg:MORspectralFactor:s1}Find $X\in\R^{n\times n}$, $X = X^\top\geq 0$ satisfying the KYP inequality $\mathcal{W}(X) \geq 0$.
		\State Compute a Cholesky-like factorization 
			\begin{displaymath}
				\mathcal{W}(X) = \left[\begin{smallmatrix}L & M\end{smallmatrix}\right]^\top \left[\begin{smallmatrix}L & M\end{smallmatrix}\right]
			\end{displaymath}
			and create the spectral factor $\Sigma_H = (A,B,L,M)$ as in~\eqref{eqn:auxiliarySystem}.
		\State\label{alg:MORspectralFactor:s3} Compute ROM $\tilde{\Sigma}_H = (\tilde{A},\tilde{B},\tilde{L},\tilde{M})$ of order $r$ of the spectral factor.
		\State\label{alg:MORspectralFactor:s4}Set $\tilde{D} \vcentcolon= \tfrac{1}{2}\tilde{M}^\top \tilde{M} + \Skew(D)$.
		\State Compute $\tilde{X}\in \R^{r\times r}$, $\tilde{X} = \tilde{X}^\top\geq 0$ satisfying the Lyapunov equation
			\begin{equation}
				\label{eqn:LyapunovSpectralFactorROM}
				\tilde{A}^\top \tilde{X} + \tilde{X}\tilde{A} + \tilde{L}^\top\tilde{L} = 0.
			\end{equation}
		\State Set $\tilde{C} \vcentcolon= \tilde{B}^\top \tilde{X} + \tilde{M}^\top \tilde{L}$.
	\end{algorithmic}
\end{algorithm}
Note that in contrast to pH-IRKA (cf.~\Cref{alg:pHIRKA}), our method is not restricted to one projection subspace and one can thus aim for $\mathcal{H}_2$-optimal MOR (w.r.t.~the spectral factors) with the intention of locally minimizing the second term in \eqref{eqn:h2_g_vs_h}.

Several remarks are in order:
\begin{enumerate}
	\item In view of \Cref{thm:H2errorSpectralFactor} it seems reasonable to choose a MOR method for the spectral factor that preserves the feedthrough matrix, i.e., $M = \tilde{M}$, since otherwise $\|H-\tilde{H}\|_{\mathcal{H}_2}$ is unbounded. This also serves as our primary motivation for the particular choice in \cref{alg:MORspectralFactor:s4} of \Cref{alg:MORspectralFactor}, which guarantees $D = \tilde{D}$ whenever $M=\tilde{M}$.  Note however, that the skew-symmetric part of $\tilde{D}$ can be chosen arbitrarily without affecting the passivity of the ROM.
	\item To ensure a positive definite solution of the Lyapunov equation~\eqref{eqn:LyapunovSpectralFactorROM}, we have to assume that the eigenvalues of $\tilde{A}$ have negative real part. If the original model~\eqref{eqn:FOM} is asymptotically stable, then this can be achieved by any MOR method that preserves asymptotic stability.  Note that the matrices in the Lyapunov equation are low-dimensional, and hence, the solution of the Lyapunov equation can be computed efficiently~\cite{Sim16}. 
	\item A pH representation of the ROM can be obtained as in \Cref{rem:pHrepresentation}. If the Lyapunov equation~\eqref{eqn:LyapunovSpectralFactorROM} is directly solved for the Cholesky or square root factors of $\tilde{X} = T^\top T$, see for instance \cite{Sim16} and the references therein, then we can perform a state-space transformation with $T$, i.e.,
	\begin{displaymath}
		\tilde{A}_T = T^{-1}\tilde{A}T^\top, \qquad \tilde{B}_T = T^{-1}\tilde{B}, \qquad \tilde{C}_T = \tilde{C}T^\top, \qquad \tilde{D}_T = \tilde{D}.
	\end{displaymath}
	Using this coordinate transformation, it is easy to see that the associated KYP inequality is solved by the identity matrix~\cite{BeaMV19}. Consequently, following \Cref{rem:pHrepresentation}, a pH realization of the form
	\begin{align*}
		\dot{\tilde{x}}_T &= (\tilde{J}_T-\tilde{R}_T)\tilde{x}_T + (\tilde{G}_T-\tilde{P}_T)u,\\
		\tilde{y} &= (\tilde{G}_T + \tilde{P}_T)^\top \tilde{x}_T + (\tilde{S}_T+\tilde{N}_T)u
	\end{align*}		
	 is obtained by setting $\tilde{J}_T \vcentcolon= \Skew(\tilde{A}_T)$, $\tilde{R}_T \vcentcolon= -\Sym(\tilde{A}_T)$, $\tilde{G}_T \vcentcolon= \tfrac{1}{2}(\tilde{C}_T^\top + \tilde{B}_T)$, $\tilde{P}_T \vcentcolon= \tfrac{1}{2}(\tilde{C}_T^\top - \tilde{B}_T)$, $\tilde{S}_T \vcentcolon= \Sym(\tilde{D}_T)$, and $\tilde{N}_T \vcentcolon= \Skew(\tilde{D}_T)$.
	\item If the ROM for the auxiliary system is constructed via projection, i.e.,  there exist matrices $V,W\in\R^{n\times r}$, satisfying $W^\top V = I_r$, and $\tilde{A} = W^\top A V$, $\tilde{B} = W^\top B$, $\tilde{L} = LV$, $\tilde{M} = M$, then the ROM obtained with \Cref{alg:MORspectralFactor} is given by 
		\begin{equation}
			\label{eqn:correctionTerm}
			\tilde{\Sigma} = (W^\top AV,W^\top B, CV + \widehat{C},D)
		\end{equation}
		with $\widehat{C} \vcentcolon= B^\top(W\tilde{X}-XV)$.
		In particular, the ROM can be decomposed as a part obtained via Petrov-Galerkin projection and some additional \emph{correction} term that ensures passivity.  While this correction term is inherent to our methodology, we remark that there are also algorithms \cite{Gri04} that construct a nearby passive system to a non-passive system by finding the smallest perturbation to the output matrix that renders the system passive. 
		
		Let us recall \cite[Thm.~3.6]{GugAB08} that the optimality conditions \eqref{h2_left_oc} for the system $(\tilde{A},\tilde{B},\tilde{L})$ are equivalent to the matrix formulation $\tilde{B}^\top \tilde{X} = B^\top Z$, where $Z$ solves the Sylvester equation 
		\begin{align*}A^\top
		 Z +Z\tilde{A} + L^\top \tilde{L}=0.
		\end{align*}
        The entire set of conditions are sometimes referred to as \emph{Wilson optimality conditions} and have initially been discussed in \cite{Wil70}. In view of these conditions, the correction term takes the form $B^\top(Z-XV)$. Due to the equations for $Z$ and $X$, this term can also be interpreted as a solution to the Sylvester equation
        \begin{align*}
         A^\top (Z-XV)+(Z-XV)\tilde{A}+XV\tilde{A}-XAV=0.
        \end{align*}
        Hence, if $A$ and $\tilde{A}$ are asymptotically stable, there exists a constant $c>0$ such that
        \begin{align*}
         \| Z-XV\| &\le c\, \| XV\tilde{A}-XAV\| \le c\, \| X\| \|(VW^\top -I)AV\|.
        \end{align*}
        In particular, if $\Ran(V)$ is $A$-invariant, we obtain $Z=XV$, and, as a consequence $\tilde{C}=CV$. Moreover, the previous bound motivates to use $X$ of small norm if the reduced model should be close to one obtained by projection. 
\end{enumerate}

\begin{theorem}
	Assume that~\eqref{eqn:FOM} is passive and asymptotically stable.  If the ROM of the auxiliary system is minimal and asymptotically stable, then the ROM constructed with \Cref{alg:MORspectralFactor} is asymptotically stable and passive.
\end{theorem}

\begin{proof}
	\emph{Asymptotic stability:} The asymptotic stability of the ROM is an immediate consequence of the asymptotic stability of the ROM for the auxiliary system.
	
	\emph{Passivity:} Since the ROM for the auxiliary system is asymptotically stable, the Lyapunov equation~\eqref{eqn:LyapunovSpectralFactorROM} has a unique solution $\tilde{X}$ and this solution is symmetric and, due to minimality of the reduced spectral factor, positive definite. By construction, we now have 
	\begin{align*}
	\tilde{W}(\tilde{X}) &= \begin{bmatrix} 
		-\tilde{A}^\top \tilde{X} - \tilde{X}\tilde{A} & \tilde{C}^\top -\tilde{X}\tilde{B} \\ 
		\tilde{C} - \tilde{B}^\top \tilde{X} & \tilde{D}+\tilde{D}^\top 
	\end{bmatrix}= \begin{bmatrix}
  		\tilde{L}^\top \tilde{L} & \tilde{L}^\top \tilde{M} \\ 
  		\tilde{M}^\top \tilde{L} & \tilde{M}^\top \tilde{M}
	\end{bmatrix} \geq 0.
	\end{align*}
	Let us define a storage function $\tilde{\mathcal{H}}\colon \mathbb R^r \to \mathbb R_{\ge 0}$ by $\tilde{\mathcal{H}}(x)=\frac{1}{2}\tilde{x}^\top \tilde{X} \tilde{x}$ and observe that, since $\dot{\tilde{x}}=\tilde{A}\tilde{x}+\tilde{B}u$, we have
	\begin{align*}
	 \frac{\dd }{\dd t}\tilde{\mathcal{H}}(\tilde{x})
	 &= \frac{1}{2} \tilde{x}^\top (\tilde{A}^\top \tilde{X}+\tilde{X}\tilde{A})x + \frac{1}{2}u^\top \tilde{B}^\top \tilde{X} \tilde{x} +\frac{1}{2}\tilde{x}^\top \tilde{X}\tilde{B}u \\
	 &=-\frac{1}{2}\tilde{x}^\top \tilde{L}^\top \tilde{L} \tilde{x} + \frac{1}{2}u^\top (\tilde{C}-\tilde{M}^\top \tilde{L}) \tilde{x}+ \frac{1}{2}\tilde{x}^\top (\tilde{C}^\top -\tilde{L}^\top \tilde{M}) u \\
	 &= -\frac{1}{2} \begin{bmatrix} \tilde{x}\\ u \end{bmatrix}^\top \begin{bmatrix} \tilde{L}^\top \tilde{L} & \tilde{L}^\top \tilde{M} \\ \tilde{M}^\top \tilde{L} & \tilde{M}^\top \tilde{M} \end{bmatrix} \begin{bmatrix} \tilde{x}\\ u \end{bmatrix}
	 + \frac{1}{2}u^\top \tilde{C}\tilde{x} + \frac{1}{2}\tilde{x}^\top \tilde{C}^\top u + \frac{1}{2}u^\top \tilde{M}^\top \tilde{M} u \\
	 &\le  \tilde{x}^\top \tilde{C}^\top u +\frac{1}{2}u^\top \tilde{M}^\top \tilde{M}u = \tilde{x}^\top \tilde{C}^\top u +u^\top \tilde{D}u=\tilde{y}^\top u
	\end{align*}
    where the last step follows since $u^\top \Skew(D)u=0$. Integration of the above inequality shows the passivity of the system $\tilde{\Sigma} = (\tilde{A},\tilde{B},\tilde{C},\tilde{D})$.
\end{proof}

In general, we cannot ensure that the ROM constructed via \Cref{alg:MORspectralFactor} is minimal, as the following example details.
\begin{example}\label{ex:no_minimality}
	Assume that the reduced spectral factor $\tilde{\Sigma}_H = (\tilde{A},\tilde{B},\tilde{L},\tilde{M})$ is given by
	\begin{align*}
		\tilde{A} &= \begin{bmatrix}
			-1 & 0\\
			2 & -2
		\end{bmatrix}, & \tilde{B} &= \begin{bmatrix}
			1\\0
		\end{bmatrix}, & \tilde{L} &= \sqrt{2}\begin{bmatrix}
			1 & -1\\
			0 & 1
		\end{bmatrix}, & \tilde{M} &= \begin{bmatrix}0\\0 \end{bmatrix}.
	\end{align*}
	Straightforward computations show that $\tilde{\Sigma}_H$ is controllable and observable, thus minimal.  We observe
	\begin{displaymath}
		\tilde{A}^\top + \tilde{A} + \tilde{L}^\top \tilde{L} = 0,
	\end{displaymath}
	implying that $\tilde{X} = I_2$ is the unique solution of~\eqref{eqn:LyapunovSpectralFactorROM}.  In view of \Cref{alg:MORspectralFactor} this implies $\tilde{C} = \tilde{B}^\top$. Nevertheless, $(\tilde{A},\tilde{C})$ is not observable, showing that $\tilde{\Sigma}$ constructed via \Cref{alg:MORspectralFactor} is not minimal.
\end{example}

\begin{remark}
  If, as in Example~\ref{ex:no_minimality}, observability is not given, the ROM constructed via \Cref{alg:MORspectralFactor} can be replaced by a minimal realization by, e.g., classical balanced truncation or subsequent truncations of Kalman controllability and observability decompositions. However, in numerical computations this may result in a loss of passivity. As an alternative, we could use the structure preserving method from \cite{BreMS20} as a post processing step as follows. Assume that $(\tilde{A},\tilde{B})$ is controllable while the Kalman observability matrix $\tilde{\mathcal{O}}:=[\tilde{C}^\top,\tilde{A}^\top \tilde{C}^\top,\dots,(\tilde{A}^{r-1})^\top \tilde{C}^\top]^\top$  satisfies $\rank(\tilde{\mathcal{O}})=k<r$. Due to \cite[Cor.~16 and Rem.~17]{BreMS20}, we can compute a passive reduced model $(\hat{A},\hat{B},\hat{C},\hat{D})$ of dimension $k$ such that $\| \tilde{G}-\hat{G}\|_{\mathcal{H}_2}=0$ and, in particular, $\tilde{C}\tilde{A}^{i}\tilde{B}=\hat{C}\hat{A}^i\hat{B},\ i\ge 0$.   Since the reduction relies on balancing the observability Gramian, we can conclude that $\rank(\tilde{\mathcal{O}})=\rank(\hat{\mathcal{O}})=k$. Denoting the individual Kalman controllability matrices by $\tilde{\mathcal{K}},\hat{\mathcal{K}}$, we now obtain 
  \begin{align*}
   \rank(\hat{\mathcal{K}}) &\ge \rank(\hat{\mathcal{O}}\hat{\mathcal{K}})=\rank(\tilde{\mathcal{O}}\tilde{\mathcal{K}})\ge \rank(\tilde{\mathcal{O}})+\rank(\tilde{\mathcal{K}})-r=\rank(\tilde{\mathcal{O}})=k.
  \end{align*}
  As a consequence, the system $(\hat{A},\hat{B},\hat{C},\hat{D})$ is controllable and observable, hence minimal.
\end{remark}

\begin{remark}
	If instead of the original system~\eqref{eqn:FOM} its dual system, given by
	\begin{equation}
		\label{eqn:dualFOM}
		\begin{aligned}
			\dot{x}_{\mathrm{d}} &= -A^\top x_{\mathrm{d}} - C^\top u_{\mathrm{d}},\qquad x_{\mathrm{d}}(0) = 0,\\
			y_{\mathrm{d}} &= B^\top x_{\mathrm{d}} + D^\top u_{\mathrm{d}},
		\end{aligned}
	\end{equation}
	is passive, then the KYP inequality~\eqref{eqn:KYPinequality} is given by
	\begin{displaymath}
		\mathcal{W}_{\mathrm{d}}(X_{\mathrm{d}}) = \begin{bmatrix}
			AX_{\mathrm{d}} + X_{\mathrm{d}}A^\top & B+X_{\mathrm{d}}C^\top\\
			B^\top + CX_{\mathrm{d}} & D + D^\top
		\end{bmatrix} \geq 0.
	\end{displaymath}
	Simple algebraic manipulations yield the factorization
	\begin{displaymath}
		\Phi(s) = \begin{bmatrix}
			(sI_n+A^\top)^{-1}(-C)^\top\\
			I_m
		\end{bmatrix}^\top\mathcal{W}_\mathrm{d}(X_{\mathrm{d}})\begin{bmatrix}
			(sI_n+A^\top)^{-1}(-C)^\top\\
			I_m
		\end{bmatrix}.
	\end{displaymath}
	In this case, we can proceed similarly as in \Cref{alg:MORspectralFactor} to construct a MOR methodology that guarantees that the dual system of the ROM is passive. 
\end{remark}

Let us emphasize that the spectral factor depends on the particular solution $X$ of the KYP inequality~\eqref{eqn:KYPinequality}. In particular, also the reduced system depends on the particular solution of the KYP inequality.  If the system at hand is given in a pH representation~\eqref{eqn:pH}, then we have seen in sub\cref{subsec:KYP} that we do not have to solve the KYP inequality in \cref{alg:MORspectralFactor:s1} of \Cref{alg:MORspectralFactor}, but can directly work with $X = Q$. Nevertheless, this choice does not guarantee that the spectral factor is particularly amendable to model-order reduction. For a related discussion for balanced truncation, we refer to \cite{BreMS20}.
Instead, one may ask if there is a particular pH representation (corresponding to a specific solution of the KYP inequality) favorable for model reduction. The following result, together with \cite[Thm.~1]{UngG19}, implies that for model reduction purposes, a particularly suitable pH representation as in \Cref{rem:pHrepresentation} is given by~$\Xmin$.

\begin{theorem}
	\label{thm:HankelSingularValuesSpectralFactor}
    Let $\Sigma=(A,B,C,D)$ be passive, minimal, and asymptotically stable. Let $\Sigma_H=(A,B,L,M)$ be the spectral factorization associated with the minimal solution $\Xmin$ of the KYP inequality~\eqref{eqn:KYPinequality}. Then, for any solution $X$ of \eqref{eqn:KYPinequality} and its spectral factorizations $\Sigma_{\hat{H}}=(A,B,\hat{L},\hat{M})$ of $\Sigma$, it holds that
    \begin{align*}
        \sigma_{k}(\Sigma_H)\le \sigma_k(\Sigma_{\hat{H}}) , \quad k=1,\dots,n,
    \end{align*}
    where $\sigma_{k}$ denotes the $k$-th Hankel singular value of $\Sigma_H$ and $\Sigma_{\hat{H}},$ respectively.
\end{theorem}

\begin{proof}
    Following \cite[Lemma 5.8]{Ant05}, the Hankel singular values of $\Sigma_H,\Sigma_{\hat{H}}$ are given by
    \begin{align*}
     \sigma_k(\Sigma_H)= \sqrt{\lambda_k(PQ)},\quad \sigma_k(\Sigma_{\hat{H}})=\sqrt{\lambda_k(\hat{P}\hat{Q})},
    \end{align*}
    where $\lambda_k$ denotes the $k$-th eigenvalue and $P,Q,\hat{P}$ and $\hat{Q}$ are the controllability and observability Gramians of $\Sigma_H$ and $\Sigma_{\hat{H}}$, respectively. Note that $P=\hat{P},Q$ and $\hat{Q}$ are the unique solutions of the Lyapunov equations
    \begin{align*}
     AP+PA^\top + BB^\top &= 0, \\
     A^\top Q+QA + L^\top L &= 0, \\
     A^\top \hat{Q}+\hat{Q}A+ \hat{L}^\top \hat{L} &= 0,
    \end{align*}
    implying that $\Xmin=Q$ and $\hat{X}=\hat{Q}$. Since $\Sigma$ is minimal, the pair $(A,B)$ is controllable and, hence, $P=\hat{P}$ is positive definite. We can thus consider its Cholesky decomposition $P=FF^\top$ and obtain 
    \begin{align*}
     (\sigma_k(\Sigma_H))^2 = \lambda_k(PQ)=\lambda_k(F^{-1}PQF)=\lambda_k(F^\top X_{\mathrm{\min}}F),\\
     (\sigma_k(\Sigma_{\hat{H}}))^2 = \lambda_k(P\hat{Q})=\lambda_k(F^{-1}P\hat{Q}F)=\lambda_k(F^\top \hat{X}F).
    \end{align*}
    The assertion now is a consequence of the Courant-Fischer-Weyl min-max principle \cite[Theorem 8.1.2]{GolV96} and the following considerations
\begin{align*}
    (\sigma_k(\Sigma_{\hat{H}}))^2 
    &= \lambda_k(F^\top \hat{X}F)
    = \min_{\substack{\mathcal{X}_k\subset \mathbb R^n\\ \dim(\mathcal{X}_k)=k}}\ \max_{\substack{z\in \mathcal{X}_k\\ \|z\|=1}} z^\top (F^\top \hat{X}F)z \\
    & \ge \min_{\substack{\mathcal{X}_k\subset \mathbb R^n\\ \dim(\mathcal{X}_k)=k}}\ \max_{\substack{z\in \mathcal{X}_k\\ \|z\|=1}} z^\top (F^\top \Xmin F)z
     =\lambda_k(F^\top \Xmin F)=(\sigma_k({\Sigma}_H))^2.
\end{align*}
\end{proof}

\begin{remark}
	\label{rem:sparsity}
	While different factorizations of the system matrix for a pH representation may result in dense system matrices $J,R$, and $Q$, we emphasize that our method works directly with the system matrix $A$, and hence, any sparsity pattern in $A$ can be exploited in the construction of the ROM for the spectral factor. In contrast, pH-IRKA (see \Cref{alg:pHIRKA}) requires matrix products with the matrix $Q$, which, depending on the particular pH representation may be computationally more or less involved in a large-scale context.
\end{remark}

\subsection*{Computational complexity}

Let us briefly discuss the computational effort required by Algorithm \ref{alg:MORspectralFactor}. For this, let us discuss the case of dense matrices, i.e., a worst-case estimate of the computational complexity. If the given Hamiltonian $Q$ is supposed to be replaced by one of the extremal solutions $X$ satisfying the KYP inequality $\mathcal{W}(X)\ge 0$ and the term $D+D^\top$ is invertible, the computations are essentially given by solving an algebraic Riccati equation. Typical dense solvers rely on an associated Hamiltonian eigenvalue problem of dimension $2n$, leading to a complexity of $\mathcal{O}(8n^3)$.  The subsequent computation of a Cholesky decomposition of the matrix $W(X)\in \mathbb R^{n+m\times n+m}$ can be realized in $\mathcal{O}((n+m)^3)$ computations. The computational effort for obtaining the reduced-order model depends on the number of iterations needed until convergence. In particular, if $\ell$ steps of an iterative procedure, such as IRKA, have to be carried out, then $2\ell n$ linear systems of equations have to be solved. For dense matrices, this will lead to a complexity of $2\ell n^3$. All further calculations depend only on the reduced system matrices and can therefore be neglected. On the other hand, large-scale systems usually result from spatial semi-discretizations of partial differential equations and thus yield sparse matrices. In this case, one can often exploit low-rank approximation procedures that scale linearly with the number of nonzero entries. For example, the solution of algebraic Riccati equations can be efficiently handled up to matrix dimension of the order $n=10^6$, see, e.g., \cite{BenS13} and the references therein.

\subsection*{Contractivity preserving model reduction}

    In context of model reduction, methods dedicated to positive real systems often come with appropriate modifications for bounded real systems and vice versa, e.g., \cite{GugA04,GuiO13,ReiS10}. Similar to condition \eqref{eqn:positiveReal}, bounded real systems $(\breve{A},\breve{B},\breve{C},\breve{D})$ are characterized by the positive (semi) definiteness of the function 
    \begin{gather*}
        \Psi\colon\C\setminus \big(\sigma(\breve{A})\cup \sigma(-\breve{A})\big) \to\C^{m\times m},\quad
        s\mapsto I_m- \breve{G}(-s)^\top \breve{G}(s)
    \end{gather*}
    on the imaginary axis and imply the contractivity of the system, i.e., we have 
    \begin{align*}
     \int_0^t \| y(\tau) \|_2^2 \dd \tau \le \int_0^t \| u(\tau)\| _2^2 \dd\tau
    \end{align*}
    for all $t>0$ and inputs $u\in L^2(0,t;\mathbb R^m)$ and $y$ associated with $(\breve{A},\breve{B},\breve{C},\breve{D})$. Following the discussion from \cref{sec:MORspectralFactorization} and the bounded real lemma \cite{AndV73}, it seems natural to aim for a factorization of $\Psi$ in terms of solutions to the bounded real Lur'e equation
    \begin{align*}
      \breve{A}^\top \breve{X} +\breve{X} \breve{A} + \breve{C}^\top \breve{C} &= -\breve{L}^\top \breve{L}\\
      \breve{X}\breve{B} + \breve{C}^\top \breve{D} &= -\breve{L}^\top \breve{M} \\
      I_m-\breve{D}^\top \breve{D} &= \breve{M}^\top \breve{M}.
    \end{align*}
    However, note that, in contrast to the positive real case, for given $(\breve{A},\breve{B},\breve{L},\breve{M})$, there is no obvious way to construct corresponding matrices $\breve{C}$ and $\breve{D}$ and an appropriate modification of \Cref{alg:MORspectralFactor} is unclear.  On the other hand, we may reverse the line of argument from \cite{ReiS10} as follows. Provided that $\mathrm{det}(I_m+\breve{G}(s))\neq 0$ and $I_m+\breve{D}$ is invertible, let us utilize the \emph{Moebius transformation} $\mathcal{M}$ defined by 
    \begin{align*}
     \breve{G}(s) \mapsto \mathcal{M}(\breve{G})(s)=G(s)=(I_m-\breve{G}(s))^{-1}(I_m+\breve{G}(s))
    \end{align*}
    which yields a positive real system $G$ that can be reduced by \Cref{alg:MORspectralFactor}. A reduced bounded real system can subsequently be constructed by application of $\mathcal{M}^{-1}$. Let us denote the associated reduced systems by $\hat{G}$ and $\hat{\breve{G}}$, respectively. Then, using the identities $\breve{G}(s)=(G(s)-I_m)(G(s)+I_m)^{-1}$ as well as $(G(s)+I_m)^{-1}=\frac{1}{2}(I_m-\breve{G}(s))$ and some algebraic manipulations similar to those in \cite{ReiS10}, we obtain 
    \begin{align*}
     \breve{G}-\hat{\breve{G}}=\frac{1}{2}(I_m-\hat{\breve{G}})(G-\hat{G})(I_m-\breve{G}).
    \end{align*}
    As a consequence, this leads to an estimate of the form
    \begin{align*}
     \| \breve{G}-\hat{\breve{G}}\| _{\mathcal{H}_2} \le \frac{1}{2} \| I_m-\hat{\breve{G}} \|_{\mathcal{H}_\infty} \| G-\hat{G}\|_{\mathcal{H}_2} \| I_m-\breve{G}\| _{\mathcal{H}_\infty}
    \end{align*}
    provided that the previous terms are all finite. Of course, in combination with \Cref{thm:H2errorSpectralFactor} one could also state an estimate in terms of the $\mathcal{H}_2$ error of the (positive real) spectral factors $H$ and $\hat{H}$.
    
    \section{Numerical examples}
\label{sec:examples}
In this section, we illustrate our theoretical findings and our novel passivity-preserving MOR method by means of numerical examples.  We emphasize that the main purpose of these examples is the illustration of the theoretical findings. An efficient implementation exploiting sparsity patterns is subject to further research. With regard to the implementation of the methods, the following remarks are in order:
\begin{itemize}
	\item To ensure a (numerically) minimal realization, we used the structure-preserving truncation algorithm presented in \cite[Sec.~5]{BreMS20} with truncation tolerance $\varepsilon_{\mathrm{trunc}} = $\num{1e-12}. The minimal realization is used to construct the ROMs. In contrast, we will use the original (numerically not minimal) realization for the error computations.
	\item For the computation of the extremal solutions of the Lur'e equation~\eqref{eqn:LureEquations}, we added an artificial feedthrough term $D+D^\top = \num{1e-12}I_m$ and constructed a solution by solving the associated Riccati equation~\eqref{eqn:ARE} using the \matlab\ built-in routine \texttt{icare}.
	\item The reduced spectral factor (step~\ref{alg:MORspectralFactor:s3} in \Cref{alg:MORspectralFactor}) is computed via IRKA (\Cref{alg:IRKA}). 
	\item For the norm computations, we used the \texttt{Control System Toolbox}.
	\item For the initialization of IRKA (\Cref{alg:IRKA}) and pH-IRKA (\Cref{alg:pHIRKA}) we compute random matrices $V,W\in\R^{n\times r}$ (for pH-IRKA only $V\in\R^{n\times r}$), construct a ROM and use this ROM to choose interpolation points and tangent directions.  In our experiments, we noticed that sometimes, the iteration got stuck in a flat local minimum. To minimize the random initialization effects, we performed the reduction \num{3} times and used only the best result (with respect to the $\mathcal{H}_2$ norm). If the ROM constructed via IRKA is not asymptotically stable, then we simply restart the iteration until an asymptotically stable ROM is constructed.
\end{itemize}

For the plot labels, we use $\Xmin$, $\Xmax$, and $Q$ to indicate if the spectral factor $\Sigma_H$ is constructed with the minimal solution of the KYP inequality, the maximal solution, or the matrix given from a direct pH modeling approach, respectively. The full-order model is denoted with FOM. The model reduction algorithms for positive real balanced truncation (\Cref{alg:prbt}) and our novel method (\Cref{alg:MORspectralFactor}) are denoted with \emph{prbt} and \emph{spectralFactor}, respectively.

To ensure reproducibility of the conducted experiments, the code for the numerical examples is publicly available under \texttt{doi} 10.5281/zenodo.4632901.

\subsection{Mass-spring-damper system}
Our first example is a multi-input multi-output mass-spring-damper system originally introduced in \cite{GugPBS12}, where the inputs describe external forces acting on the first two masses. The outputs are chosen as the corresponding velocities, thus rendering the system passive. In fact, the model is directly given in port-Hamiltonian form~\eqref{eqn:pH} with $P=0$, $S=N=0$. For details on the setup of the system matrices, we refer to \cite{GugPBS12}. To demonstrate our methods, we use a moderate system dimension of $n=1000$.  The numerically minimal pH realization obtained with the algorithm from \cite[Sec.~5]{BreMS20} yields a system of dimension $n=86$, with an $\mathcal{H}_2$ error of~\num{6.6588e-07} and an $\mathcal{H}_\infty$ error of~\num{1.8571e-06}. 

\begin{figure}
%
%
\begin{tikzpicture}

\begin{axis}[%
width=5in,
height=1.8in,
at={(1.011in,0.642in)},
scale only axis,
grid=both,
grid style={line width=.1pt, draw=gray!10},
major grid style={line width=.2pt,draw=gray!50},
axis lines*=left,
axis line style={line width=\lineWidth},
mark size=2.2pt,
xmin=1,
xmax=80,
ymode=log,
ymin=1e-9,
ymax=1,
yminorticks=true,
axis background/.style={fill=white},
legend style={%
	legend cell align=left, 
	align=left, 
	font=\tiny,
	draw=white!15!black,
	at={(0.01,0.02)},
	anchor=south west,},
]
\addplot [color=mycolor1,line width=\lineWidth,mark=triangle*,mark repeat={4}]
  table[row sep=crcr]{%
1	0.262658410927011\\
2	0.24676837395542\\
3	0.209182559038404\\
4	0.143757359180192\\
5	0.0538952054669837\\
6	0.026108111265548\\
7	0.0190207929977854\\
8	0.00931342822641115\\
9	0.00573421550964604\\
10	0.00377501293232087\\
11	0.00184056160764367\\
12	0.00171049433849583\\
13	0.000732400875694053\\
14	0.000477725528445029\\
15	0.000205513910596189\\
16	0.000151285468838474\\
17	6.49189814301029e-05\\
18	4.12031489864278e-05\\
19	1.54973771670323e-05\\
20	1.32433176215097e-05\\
21	4.54943170146516e-06\\
22	2.93038814753891e-06\\
23	1.74836376982732e-06\\
24	1.69153644655503e-06\\
25	1.66297736885675e-06\\
26	1.53663559889858e-06\\
27	1.46706031448132e-06\\
28	1.40105491750388e-06\\
29	1.29169753936329e-06\\
30	1.28218433906231e-06\\
31	1.15989498496129e-06\\
32	1.13592100045554e-06\\
33	1.03503742414653e-06\\
34	9.98827642380759e-07\\
35	9.17574900776297e-07\\
36	8.78938954517078e-07\\
37	8.11509848217717e-07\\
38	7.74633499264375e-07\\
39	7.17633074167926e-07\\
40	6.8441756742355e-07\\
41	6.35666729929502e-07\\
42	6.07002343224579e-07\\
43	5.65029591606709e-07\\
44	5.41254283632807e-07\\
45	5.05082565908598e-07\\
46	4.86081582472087e-07\\
47	4.55164712895714e-07\\
48	4.40334271174131e-07\\
49	4.14528071574376e-07\\
50	4.02845382655373e-07\\
51	3.82240241593326e-07\\
52	3.72584736188783e-07\\
53	3.57120809804561e-07\\
54	3.48662019869399e-07\\
55	3.3778645452146e-07\\
56	3.30109300351816e-07\\
57	3.22800092654408e-07\\
58	3.15675913793721e-07\\
59	3.10759674480793e-07\\
60	3.03778520624702e-07\\
61	3.00236380033199e-07\\
62	2.92589005771338e-07\\
63	2.89709759227804e-07\\
64	2.802164775156e-07\\
65	2.77604949273897e-07\\
66	2.64871528690753e-07\\
67	2.62357993423343e-07\\
68	2.44901017287436e-07\\
69	2.42433587603214e-07\\
70	2.18729160288406e-07\\
71	2.16305363939269e-07\\
72	1.84806256469738e-07\\
73	1.82450843037578e-07\\
74	1.41683063074132e-07\\
75	1.39465288731045e-07\\
76	1.10750728635128e-07\\
77	8.84939898762005e-08\\
78	8.66325432063652e-08\\
79	2.71431409237294e-08\\
80	2.60139111766211e-08\\
81	9.50331085064568e-09\\
82	4.98213788792087e-09\\
83	5.35067205320636e-10\\
84	5.03566454539682e-10\\
85	2.37392170213593e-10\\
86	4.9043162393526e-11\\
};
\addlegendentry{FOM}

\addplot [color=mycolor2,line width=\lineWidth,mark=square*,mark repeat={4}]
  table[row sep=crcr]{%
1	0.622485749254663\\
2	0.593917797451878\\
3	0.580591454724331\\
4	0.499940474050145\\
5	0.382574709481671\\
6	0.349041546448508\\
7	0.312584077307488\\
8	0.285054613333661\\
9	0.237262330430538\\
10	0.208504448477484\\
11	0.179298356138511\\
12	0.154597067550639\\
13	0.134837204841838\\
14	0.116229852835224\\
15	0.10150083992919\\
16	0.0881587584129547\\
17	0.0769903250743888\\
18	0.0682583167328069\\
19	0.0601710155006094\\
20	0.052974492227516\\
21	0.0468328347304987\\
22	0.0409647046418102\\
23	0.0356419603969933\\
24	0.0309851743473521\\
25	0.0268915290387614\\
26	0.0234224399931709\\
27	0.0202314372027215\\
28	0.0174135668054045\\
29	0.0150463517602438\\
30	0.0130109318625344\\
31	0.0113016959601459\\
32	0.00973868326345961\\
33	0.00834987830765905\\
34	0.00717343470596063\\
35	0.00618523913943937\\
36	0.00540134756002302\\
37	0.00468800621795635\\
38	0.0040180618142731\\
39	0.00341904315664821\\
40	0.00288864560524466\\
41	0.00244676417828863\\
42	0.00207009683514385\\
43	0.00178454323036697\\
44	0.00158862904514602\\
45	0.00139223943575268\\
46	0.00118915729822454\\
47	0.0010028553568882\\
48	0.000847791139322931\\
49	0.00071281466103715\\
50	0.000601857184802994\\
51	0.000506671603253895\\
52	0.000428932037294678\\
53	0.000365706431291032\\
54	0.000318320962159796\\
55	0.000281029401779505\\
56	0.000241617249759118\\
57	0.000205568408817995\\
58	0.000173693540706202\\
59	0.000146710469277197\\
60	0.00012412479471977\\
61	0.000104789623392516\\
62	8.92289491967703e-05\\
63	7.56920415949748e-05\\
64	6.53838229237788e-05\\
65	5.62726527332316e-05\\
66	4.87548169321921e-05\\
67	4.18300443618629e-05\\
68	3.54959781973769e-05\\
69	3.03289769180277e-05\\
70	2.54312819071367e-05\\
71	2.17561310772506e-05\\
72	1.81007992921896e-05\\
73	1.55202155158764e-05\\
74	1.28180666019869e-05\\
75	1.10222238718751e-05\\
76	9.02450513111276e-06\\
77	7.79600241666899e-06\\
78	6.30194131345902e-06\\
79	5.48833473121883e-06\\
80	4.33839676744679e-06\\
81	3.82671200408629e-06\\
82	2.90086322979378e-06\\
83	2.60117211731069e-06\\
84	1.81460827228544e-06\\
85	1.66019211574907e-06\\
86	1.36901233934319e-06\\
};
\addlegendentry{$\Sigma_H(Q)$}

\addplot [color=mycolor3,line width=\lineWidth,mark=*,mark repeat={4}]
  table[row sep=crcr]{%
1	0.906999346308722\\
2	0.906985877578904\\
3	0.903735868008308\\
4	0.903688695358915\\
5	0.89665766997968\\
6	0.896369465168959\\
7	0.883823497976038\\
8	0.883213361987064\\
9	0.865328038378719\\
10	0.864868130905418\\
11	0.84126395018243\\
12	0.840440772424267\\
13	0.813564546196852\\
14	0.811516862048184\\
15	0.783629321016457\\
16	0.783560535774649\\
17	0.775493026907214\\
18	0.775478500469093\\
19	0.771764907628305\\
20	0.771677014118225\\
21	0.766177186474301\\
22	0.765368595921433\\
23	0.75857693636517\\
24	0.757170110364851\\
25	0.754241041853005\\
26	0.751723730080974\\
27	0.748431241010241\\
28	0.745709062265399\\
29	0.735412096796004\\
30	0.734966520737886\\
31	0.72323501970323\\
32	0.722385928675969\\
33	0.722282359967371\\
34	0.718042344331448\\
35	0.708322899036233\\
36	0.704572476752583\\
37	0.694867393179296\\
38	0.693589553660709\\
39	0.68796924510009\\
40	0.687019743661886\\
41	0.676785843153119\\
42	0.669060320535855\\
43	0.666945316958664\\
44	0.66562119324669\\
45	0.654964419219927\\
46	0.652595010656001\\
47	0.643336703565868\\
48	0.640947577725554\\
49	0.640440023139311\\
50	0.632491798050672\\
51	0.62612706656701\\
52	0.618722001574195\\
53	0.618042511444757\\
54	0.615593605763574\\
55	0.608419182022092\\
56	0.607556316482067\\
57	0.601063372696055\\
58	0.600667593125837\\
59	0.596417425028283\\
60	0.593999248246778\\
61	0.590309576424819\\
62	0.589714761735426\\
63	0.588798089311165\\
64	0.586726882145084\\
65	0.586289811646713\\
66	0.569425238202234\\
67	0.542785477227041\\
68	0.523235120874508\\
69	0.490128989008463\\
70	0.476833912114081\\
71	0.432217690853525\\
72	0.42333776690693\\
73	0.389337501602807\\
74	0.350497539665007\\
75	0.345697929322875\\
76	0.309636648208629\\
77	0.276926245554952\\
78	0.271227505327923\\
79	0.232627578114028\\
80	0.223403164312975\\
81	0.191850495987186\\
82	0.16376939739611\\
83	0.144998004307407\\
84	0.0971433404739447\\
85	0.059677465186822\\
86	0.0320312494238117\\
};
\addlegendentry{$\Sigma_H(\Xmax)$}

\addplot [color=mycolor4,line width=\lineWidth,mark=diamond*,mark repeat={4}]
  table[row sep=crcr]{%
1	0.588058858468592\\
2	0.543409571883338\\
3	0.434126638698413\\
4	0.270199035267138\\
5	0.154698742899607\\
6	0.0901678411717945\\
7	0.0558681398611132\\
8	0.0509703429282382\\
9	0.022480806180874\\
10	0.0154047287887463\\
11	0.00682266752685747\\
12	0.00495871391548999\\
13	0.00231808427216914\\
14	0.00140144194989784\\
15	0.000602130532749227\\
16	0.000434621279595914\\
17	0.000181967107329234\\
18	0.000116559346454159\\
19	3.95459980382494e-05\\
20	3.79923392008105e-05\\
21	1.21945458318574e-05\\
22	7.74757851342047e-06\\
23	3.74599716090626e-06\\
24	3.40995795993037e-06\\
25	2.52293502441833e-06\\
26	2.32603268057825e-06\\
27	2.07231538897852e-06\\
28	1.94828769512135e-06\\
29	1.79588466354479e-06\\
30	1.70350543806203e-06\\
31	1.58782137771724e-06\\
32	1.5179790921538e-06\\
33	1.42003116629778e-06\\
34	1.36956874902316e-06\\
35	1.28034490693812e-06\\
36	1.24727649618118e-06\\
37	1.16179726107947e-06\\
38	1.14312347865169e-06\\
39	1.05974523760178e-06\\
40	1.05104418039342e-06\\
41	9.70847721585598e-07\\
42	9.67769646325323e-07\\
43	8.92597836883757e-07\\
44	8.92033995123253e-07\\
45	8.2311355385423e-07\\
46	8.23080698949967e-07\\
47	7.60845365819716e-07\\
48	7.60256069939342e-07\\
49	7.04850016279062e-07\\
50	7.02655181126654e-07\\
51	6.54526229584304e-07\\
52	6.49514575725319e-07\\
53	6.09945423586807e-07\\
54	6.00087592891206e-07\\
55	5.7113775345777e-07\\
56	5.53688143016095e-07\\
57	5.35071876310229e-07\\
58	5.09686594795073e-07\\
59	4.97869100608015e-07\\
60	4.67485999398121e-07\\
61	4.59514635610954e-07\\
62	4.26495032518715e-07\\
63	4.20858166877169e-07\\
64	3.86127073330711e-07\\
65	3.82044768732601e-07\\
66	3.64586202661925e-07\\
67	3.45535159389487e-07\\
68	3.42756029190708e-07\\
69	3.04294308970433e-07\\
70	3.02405185641843e-07\\
71	2.61283457494489e-07\\
72	2.60163858841856e-07\\
73	2.15336396614275e-07\\
74	2.14889996531147e-07\\
75	1.64971298067421e-07\\
76	1.64810854377589e-07\\
77	1.0806224172541e-07\\
78	1.07356703005577e-07\\
79	4.02446088617002e-08\\
80	3.90347917228112e-08\\
81	1.67858811238755e-08\\
82	1.52114905940233e-08\\
83	9.89442667925698e-10\\
84	8.94121080312384e-10\\
85	4.08545869832451e-10\\
86	1.15516557447964e-10\\
};
\addlegendentry{$\Sigma_H(\Xmin)$}

\end{axis}
\end{tikzpicture}%
	\caption{Hankel singular values for the full-order model (FOM) of the mass-spring-damper system and the auxiliary system $\Sigma_H$ for different solutions of the KYP-inequality}
	\label{fig:msd-n1000-HankelSingularValues}
\end{figure}
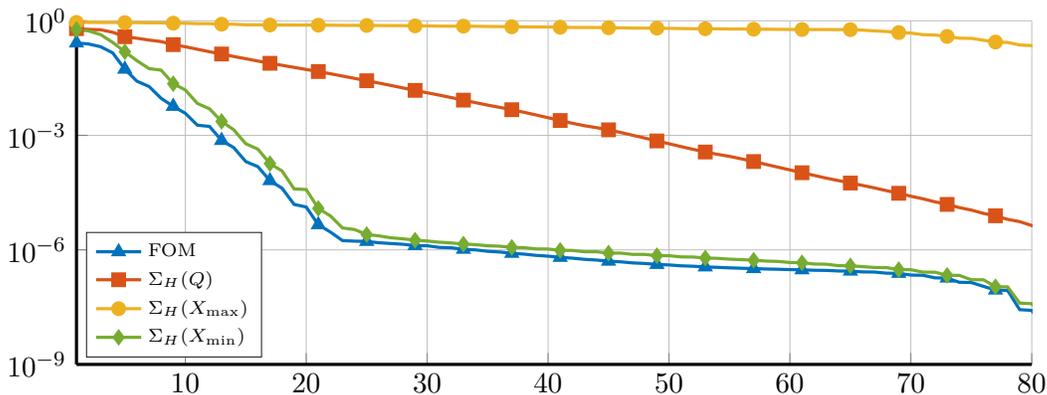

The Hankel singular values for the FOM and the spectral factors corresponding to different pH realizations are presented in Figure~\ref{fig:msd-n1000-HankelSingularValues}. We observe that the decay of the singular values corresponding to the spectral factor for $\Xmin$ (green diamonds) is quite similar to the decay of the FOM (blue triangles). In contrast, the singular values for the realization corresponding to the matrix $Q$ (red squares) coming directly from the model, and corresponding to $\Xmax$ (yellow circles) have a much slower decay, which is in agreement with \Cref{thm:HankelSingularValuesSpectralFactor}. In particular, we expect MOR methods working with these realizations to have a more significant approximation error. This is indeed the case, as we showcase in Figure~\ref{fig:msd-n1000-Errors}. In particular, the ROM constructed with our novel MOR method based on spectral factorization (cf.~\Cref{alg:MORspectralFactor}) corresponding to $\Xmin$ has a similar error to the $\mathcal{H}_2$ optimal approximation obtained with IRKA, while at the same time ensuring passivity of the ROM.  While our method performs consistently better than pH-IRKA, we want to emphasize that also pH-IRKA depends strongly on the specific realization. Indeed, for reduced dimension $r=16$, the pH-IRKA ROM corresponding to $\Xmin$ has an $\mathcal{H}_2$ error almost 4 magnitudes better than the pH-IRKA ROM corresponding to the original energy~$Q$. 

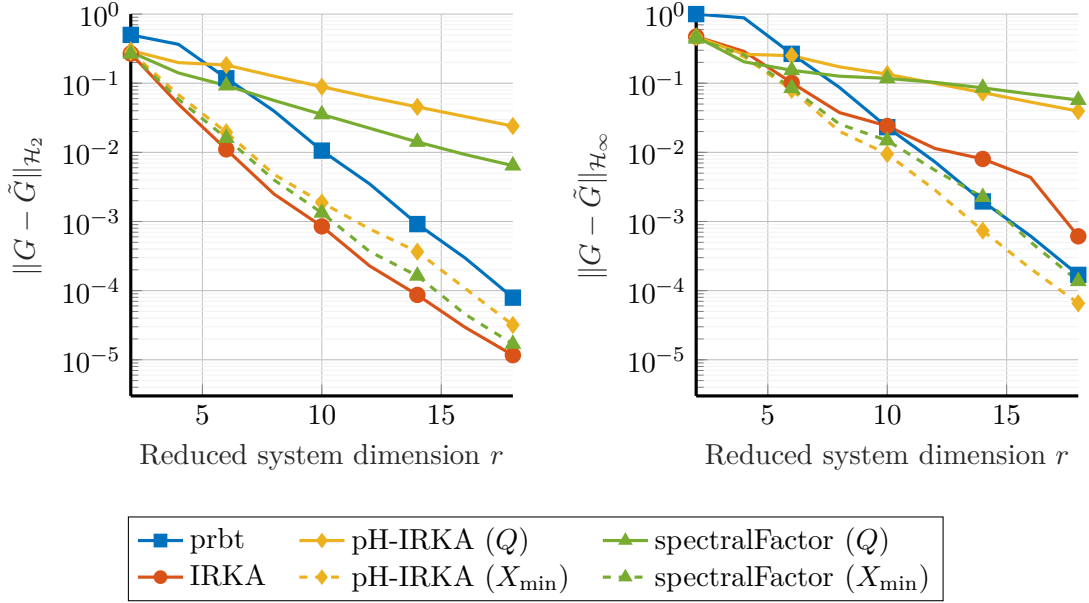
\begin{figure}[t]
	\centering
	\begin{subfigure}[t]{.45\linewidth}
%
\begin{tikzpicture}

\begin{axis}[%
width=2in,
height=2in,
at={(1.011in,0.642in)},
scale only axis,
grid=both,
grid style={line width=.1pt, draw=gray!10},
major grid style={line width=.2pt,draw=gray!50},
axis lines*=left,
axis line style={line width=\lineWidth},
xmin=2,
xmax=18,
xlabel style={font=\color{white!15!black}},
xlabel={Reduced system dimension $r$},
ymode=log,
ymin=3e-06,
ymax=1,
yminorticks=true,
ylabel style={font=\color{white!15!black}},
ylabel={$\|G-\tilde{G}\|_{\mathcal{H}_2}$},
axis background/.style={fill=white},
mark size=2.5pt,
legend style={
	at={(0.5,-0.1)},
	anchor=north,
	legend cell align={left},
	/tikz/column 2/.style={
                column sep=10pt,
            },
    /tikz/column 4/.style={
                column sep=10pt,
            },
},
legend columns=3,
legend to name=named, 
]
\addplot [color=mycolor1,line width=\lineWidth,mark=square*,mark repeat={2}]
  table[row sep=crcr]{%
2	0.500142075842316\\
4	0.365099847528314\\
6	0.118077069228029\\
8	0.0395443244870562\\
10	0.0106013687430278\\
12	0.00350047111165301\\
14	0.00091763345874702\\
16	0.000296623737492816\\
18	7.9164701370152e-05\\
20	2.38300449227502e-05\\
};
\addlegendentry{prbt}

\addplot [color=mycolor3,line width=\lineWidth,mark=diamond*,mark repeat={2}]
  table[row sep=crcr]{%
2	0.2974177609813\\
4	0.198643596181537\\
6	0.183112248478162\\
8	0.126402841439287\\
10	0.0890206462006391\\
12	0.0628049026652457\\
14	0.0454569935419281\\
16	0.0330075978384858\\
18	0.023992226356763\\
20	0.0167870670303922\\
};
\addlegendentry{pH-IRKA ($Q$)}

\addplot [color=mycolor4,line width=\lineWidth,mark=triangle*,mark repeat={2}]
  table[row sep=crcr]{%
2	0.282049095162976\\
4	0.140741686254106\\
6	0.0924985619278518\\
8	0.05628512292357\\
10	0.0354117403639371\\
12	0.0223430694038237\\
14	0.0141360933855319\\
16	0.00930515575937626\\
18	0.00643078563712058\\
20	0.00452277379185621\\
};
\addlegendentry{spectralFactor ($Q$)}

\addplot [color=mycolor2,line width=\lineWidth,mark=*,mark repeat={2}]
  table[row sep=crcr]{%
2	0.268099760418536\\
4	0.0496202070172605\\
6	0.0110460475130039\\
8	0.00249165713888\\
10	0.000850235843846243\\
12	0.000228104618743837\\
14	8.66576349696459e-05\\
16	2.94266692091679e-05\\
18	1.16489020109264e-05\\
20	2.86234259082885e-06\\
};
\addlegendentry{IRKA}

\addplot [color=mycolor3, dashed,line width=\lineWidth,mark=diamond*,mark repeat={2},mark options={solid}]
  table[row sep=crcr]{%
2	0.283551647113907\\
4	0.0679575701437827\\
6	0.0194828347867743\\
8	0.00479921702069979\\
10	0.00187667263460867\\
12	0.000780920940573837\\
14	0.000365462153284192\\
16	0.000107756802626085\\
18	3.1836033811607e-05\\
20	8.58976290486202e-06\\
};
\addlegendentry{pH-IRKA ($\Xmin$)}

\addplot [color=mycolor4, dashed,line width=\lineWidth,mark=triangle*,mark repeat={2},mark options={solid}]
  table[row sep=crcr]{%
2	0.266212568435337\\
4	0.0583894564043145\\
6	0.0160718122627577\\
8	0.00398921678756003\\
10	0.00133583400297092\\
12	0.000368284809895793\\
14	0.000163755065721828\\
16	4.55380432339032e-05\\
18	1.69213587668303e-05\\
20	4.11321756664272e-06\\
};
\addlegendentry{spectralFactor ($\Xmin$)}
\end{axis}
\end{tikzpicture}
	\end{subfigure}\qquad
	\begin{subfigure}[t]{.45\linewidth}
%
\begin{tikzpicture}

\begin{axis}[%
width=2in,
height=2in,
at={(1.011in,0.642in)},
scale only axis,
grid=both,
grid style={line width=.1pt, draw=gray!10},
major grid style={line width=.2pt,draw=gray!50},
axis lines*=left,
axis line style={line width=\lineWidth},
xmin=2,
xmax=18,
xlabel style={font=\color{white!15!black}},
xlabel={Reduced system dimension $r$},
ymode=log,
ymin=3e-06,
ymax=1,
yminorticks=true,
ylabel style={font=\color{white!15!black}},
ylabel={$\|G-\tilde{G}\|_{\mathcal{H}_{\infty}}$},
axis background/.style={fill=white},
mark size=2.5pt,
legend style={legend cell align=left, align=left, draw=white!15!black}
]
\addplot [color=mycolor1,line width=\lineWidth,mark=square*,mark repeat={2}]
  table[row sep=crcr]{%
2	0.999999999986881\\
4	0.883010568527452\\
6	0.266207805627651\\
8	0.0869556857003491\\
10	0.0231153009033602\\
12	0.00727284377404146\\
14	0.00195228169554109\\
16	0.000616816700911946\\
18	0.000168992084500852\\
20	5.05716782435094e-05\\
};

\addplot [color=mycolor2,line width=\lineWidth,mark=*,mark repeat={2}]
  table[row sep=crcr]{%
2	0.473126834201283\\
4	0.288669352707014\\
6	0.101545464312111\\
8	0.037622320943323\\
10	0.024071329113302\\
12	0.0113679905494557\\
14	0.00802903357226395\\
16	0.00435301535925433\\
18	0.000610150874766629\\
20	8.93145040118393e-05\\
};

\addplot [color=mycolor3,line width=\lineWidth,mark=diamond*,mark repeat={2}]
  table[row sep=crcr]{%
2	0.468952382126651\\
4	0.262267389534601\\
6	0.250264651818297\\
8	0.172602591391689\\
10	0.135156912690854\\
12	0.0995374461481257\\
14	0.0732017149292287\\
16	0.0534210176239147\\
18	0.0394910741747845\\
20	0.0269757741167765\\
};

\addplot [color=mycolor3, dashed,line width=\lineWidth,mark=diamond*,mark repeat={2},mark options={solid}]
  table[row sep=crcr]{%
2	0.466395805564058\\
4	0.240824304302034\\
6	0.0802276077085316\\
8	0.0201740140497714\\
10	0.00944495912615902\\
12	0.00284964508260534\\
14	0.000738464444750734\\
16	0.00020828624326545\\
18	6.55506335030448e-05\\
20	1.72692844127286e-05\\
};

\addplot [color=mycolor4,line width=\lineWidth,mark=triangle*,mark repeat={2}]
  table[row sep=crcr]{%
2	0.469295302758289\\
4	0.204216723419618\\
6	0.15401023268703\\
8	0.126366046514569\\
10	0.117512565087252\\
12	0.102901828977595\\
14	0.0856610422996424\\
16	0.0696257511646717\\
18	0.0570700959741421\\
20	0.0452478104992506\\
};

\addplot [color=mycolor4, dashed,line width=\lineWidth,mark=triangle*,mark repeat={2},mark options={solid}]
  table[row sep=crcr]{%
2	0.440229421192609\\
4	0.256512888186844\\
6	0.0850362783222356\\
8	0.0256071667312906\\
10	0.0148305365630362\\
12	0.00551383438600766\\
14	0.00224884274123628\\
16	0.000505452618137634\\
18	0.000136698489381655\\
20	3.02641104785606e-05\\
};

\end{axis}
\end{tikzpicture}%
	\end{subfigure}\\
	\hypersetup{linkcolor=black}
	\ref{named} 
	\caption{Error for the mass-spring-damper system. Left: $\mathcal{H}_\infty$ error. Right: $\mathcal{H}_2$ error.}
	\label{fig:msd-n1000-Errors}
\end{figure}

We conclude our example with a quick investigation of the a-posteriori error bound presented in \Cref{thm:H2errorSpectralFactor}.   The errors $\|G-\tilde{G}\|$ for the system and $\|H-\tilde{H}\|$ for the spectral factors, are presented in \Cref{tab:msd-n1000-ErrorsSpectralFactors}. 
%
%
\begin{table}
	\centering
	\caption{Approximation errors for the system ($G$ vs.~$\tilde{G}$) and the spectral factors ($H$ vs.~$\tilde{H}$) for the mass-spring-damper system}
	\label{tab:msd-n1000-ErrorsSpectralFactors}
	\begin{tabular}{c@{\hspace{1.5em}}cc@{\hspace{1.5em}}cc}
		\toprule
		& \multicolumn{2}{c}{spectralFactor $(Q)$} &  \multicolumn{2}{c}{spectralFactor $(\Xmin)$}\\\cmidrule(r{1.5em}){2-3}\cmidrule{4-5}
		$r$ & $\|G-\tilde{G}\|_{\mathcal{H}_2}$ & $\|H-\tilde{H}\|_{\mathcal{H}_2}$ & $\|G-\tilde{G}\|_{\mathcal{H}_2}$ & $\|H-\tilde{H}\|_{\mathcal{H}_2}$\\\midrule
$4$& \num{1.407e-01}& \num{3.892e-01}& \num{5.839e-02}& \num{9.943e-02}\\
$8$& \num{5.629e-02}& \num{2.180e-01}& \num{3.989e-03}& \num{7.258e-03}\\
$12$& \num{2.234e-02}& \num{1.164e-01}& \num{3.683e-04}& \num{8.248e-04}\\
$16$& \num{9.305e-03}& \num{6.176e-02}& \num{4.554e-05}& \num{1.007e-04}\\
\bottomrule
	\end{tabular}
\end{table}
We notice that the behavior of the error of the system approximation, i.e., $\|G-\tilde{G}\|_{\mathcal{H}_2}$, is similar to the behavior of the approximation quality for the spectral factor, i.e., $\|H-\tilde{H}\|_{\mathcal{H}_2}$.  For the realization based on $\Xmin$, the error in the spectral factor $\|H-\tilde{H}\|_{\mathcal{H}_2}$ is in close agreement to the error in the system. For the realization based on the original $Q$, this difference is bigger.

\begin{remark}
	In our numerical experiment, we observed that the norm of the correction term reported in~\eqref{eqn:correctionTerm}, i.e., the norm of the correction system
	\begin{displaymath}
		\widehat{G}(s) = \widehat{C}(sI_r-\tilde{A})^{-1}\tilde{B}
	\end{displaymath}
	has a similar decay as the error system $G-\tilde{G}$.  Consequently, for the spectral factor for $\Xmin$, the ROM is close to a ROM obtained via projection.
\end{remark}


\subsection{Passive linear poroelasticity}
For our second example, we consider Biot's consolidation model for poroelasticity \cite{Bio41}, which describes the deformation of a porous material fully saturated by a viscous fluid. For a bounded Lipschitz domain $\Omega\subseteq\R^d$ with $d\in\{2,3\}$ and a time interval $\mathbb{T} \vcentcolon= [0,T]$, one wants to determine the displacement field $u\colon\mathbb{T}\times\Omega\to\R^d$ for the porous material and the pressure $p\colon\mathbb{T}\times\Omega\to\R$ for the viscous fluid satisfying the coupled hyperbolic-parabolic PDE
\begin{subequations}
	\label{eqn:poro}
	\begin{align}
		\label{eqn:poro:u}\rho\partial_{tt}u -\nabla  \sigma(u) + \nabla (\alpha p) &= f,\\
		\label{eqn:poro:p}\partial_t\big( \alpha\nabla \cdot u + \tfrac{1}{M} p\big) - \nabla \cdot\big(\tfrac{\kappa}{\nu}\nabla p\big) &= g,
	\end{align}
\end{subequations}
with stress-strain constitute relation
\begin{displaymath}
	\sigma(u) = 2\mu\varepsilon(u) + \lambda(\nabla\cdot u)\mathcal{I},\quad \varepsilon(u) = \tfrac{1}{2}\big(\nabla u + (\nabla u)^\top).
\end{displaymath}
Hereby,  $\mu$ and $\lambda$ are the Lam\'{e} coefficients, and $\mathcal{I}$ is the identity tensor. The quantities $\alpha$, $M$, $\kappa$, $\rho$, $\nu$, $f$ and~$g$ denote the Biot-Willis fluid-solid coupling coefficient, Biot modulus, permeability, density, fluid viscosity, volume-distributed forces, and external injection, respectively. After a first-order reformulation, a finite-element discretization with standard $P_1$ Lagrange finite element spaces for the associated weak formulation (which we perform with the python interface fenics), and using the generalized state-space pH formulation (cf.~Remark\ref{rem:pHdescriptor}) derived in \cite{AltMU20c}, we obtain the linear system of equations
\begin{displaymath}
	E\dot{x} = (J-R)x + Bv,\qquad y= B^\top x
\end{displaymath}
with 
\begin{gather*}
	E \vcentcolon= \begin{bmatrix}
		\rho M_u & 0 & 0\\
		0 & K_u(\mu,\lambda) & 0\\
		0 & 0 & \frac{1}{M} M_p
	\end{bmatrix},  \quad
	J \vcentcolon= \begin{bmatrix}
		0 & -K_u(\mu,\lambda) & \alpha D^\top\\
		K_u(\mu,\lambda) & 0 & 0\\
		-\alpha D & 0 & 0
	\end{bmatrix},\\	 
	R \vcentcolon= \begin{bmatrix}
		0 & 0 & 0\\
		0 & 0 & 0\\
		0 & 0 & \tfrac{\kappa}{\nu} K_p
	\end{bmatrix},  \quad 
	B \vcentcolon= \begin{bmatrix}
		B_f & 0\\
		0 & 0\\
		0 & B_g
	\end{bmatrix}, \quad 
	x \vcentcolon=\begin{bmatrix}
		w_h\\
		u_h\\
		p_h
	\end{bmatrix}.
\end{gather*}
Let us emphasize that $E$ is a symmetric positive definite mass matrix. Due to the hyperbolic character of~\eqref{eqn:poro:u}, the system has eigenvalues on the imaginary axis, which is difficult for MOR.  Although MOR methods for hyperbolic systems or systems with slowly decaying Hankel singular values are subject to extensive research (for an overview we refer to \cite[Sec.~2]{BlaSU20}), we avoid this issue by adding artificial damping, resulting in the model
\begin{displaymath}
	E\dot{x} = (J-R-\eta I_n)x + Bv,\qquad y= B^\top x.
\end{displaymath}
For our experiment we use the unit square $\Omega = [0,1]^2$ with $d=2$, homogeneous Dirichlet boundary conditions,  spatially independent volume-distributed forces $f$ and injection $g$ (yielding $m=2$), and the parameters listed in \Cref{tab:poro:parameterSetting}.
\begin{table}
	\centering
	\caption{Numerical parameters for the poroelastic system}
	\label{tab:poro:parameterSetting}
	\begin{tabular}{ccccccc}
		\toprule
		$\mu$ & $\lambda$ & $\rho$ & $\alpha$ & $\tfrac{1}{M}$ & $\tfrac{\kappa}{\nu}$ & $\eta$\\\midrule
		\num{12} & \num{6} & \num{1e-3} & \num{0.79} & \num{7.80e3} & \num{633.33} & \num{1e-3}\\\bottomrule
	\end{tabular}
\end{table}
The resulting system has dimension $n=980$. To convert the system to the standard pH representation, as introduced in \eqref{eqn:pH}, we perform the coordinate transformation $z = Ex$, resulting in $Q = E^{-1}$. We emphasize that in practice, one should directly work with the generalized state-space representation (cf.~\Cref{rem:pHdescriptor}). However, for the sake of consistency, we proceed with the representation~\eqref{eqn:pH}. After computing a structure-preserving minimal realization we obtain a system of dimension $n=83$ with a relative $\mathcal{H}_2$ error of \num{5.4580e-05} and a relative $\mathcal{H}_\infty$ error of \num{3.8634e-05}.

The decay of the Hankel singular values for the FOM and the spectral factor $\Sigma_H$ for different solutions of the KYP-inequality is presented in Figure~\ref{fig:poro-n980-HankelSingularValues} for different values of the artificial damping parameter $\eta$. 
\begin{figure}
	\sisetup{retain-unity-mantissa = false}
	\input{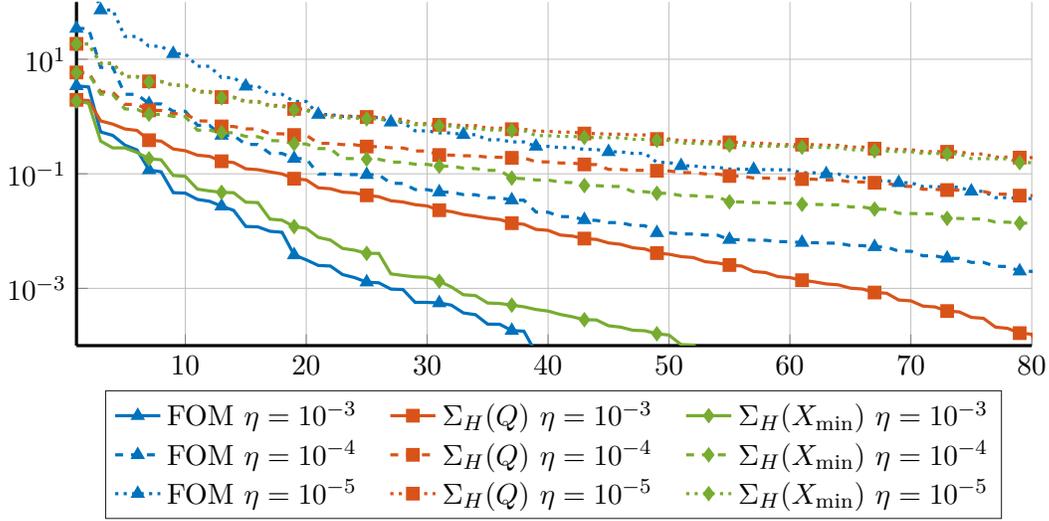}
	\caption{Hankel singular values for the full-order model (FOM) for the poroelasticity system and the auxiliary system $\Sigma_H$ for different solutions of the KYP-inequality and different artificial damping parameters $\eta$}
	\label{fig:poro-n980-HankelSingularValues}
\end{figure}
While \Cref{thm:HankelSingularValuesSpectralFactor} is true independent of the particular choice of the artificial damping parameter $\eta=\num{1e-3}$ (solid), $\eta=\num{1e-4}$ (dashed), $\eta=\num{1e-5}$ (dotted), for larger damping parameter we observe a stronger decay of the Hankel singular values and notice, that for a larger damping parameter also the difference between the Hankel singular values for $Q$ and $\Xmin$ increases. 

The $\mathcal{H}_2$ error between the FOM and the different model reduction schemes is presented in Figure~\ref{fig:poro-n980-H2error}.
\begin{figure}
	\centering
%
%
\begin{tikzpicture}

\begin{axis}[%
width=3.3in,
height=2in,
at={(1.011in,0.642in)},
scale only axis,
grid=both,
grid style={line width=.1pt, draw=gray!10},
major grid style={line width=.2pt,draw=gray!50},
axis lines*=left,
axis line style={line width=\lineWidth},
xmin=2,
xmax=18,
xlabel style={font=\color{white!15!black}},
xlabel={\small Reduced system dimension $r$},
ymode=log,
ymin=1e-01,
ymax=2e2,
yminorticks=true,
ylabel style={font=\color{white!15!black}},
ylabel={$\|G-\tilde{G}\|_{\mathcal{H}_2}$},
axis background/.style={fill=white},
mark size=2.2pt,
legend style={legend cell align=left, align=left, draw=white!15!black},
legend pos=outer north east,
]
\addplot [color=mycolor1,line width=\lineWidth,mark=square*,mark repeat={1}]
  table[row sep=crcr]{%
2	54.0741439014348\\
4	15.4253190847486\\
6	12.4137452723373\\
8	13.4737971958988\\
10	15.2876224638548\\
12	3.47378461765405\\
14	1.78013354317897\\
16	1.62097739838108\\
18	0.93389306981479\\
20	0.432864680544499\\
};
\addlegendentry{prbt}

\addplot [color=mycolor2,line width=\lineWidth,mark=*,mark repeat={1}]
  table[row sep=crcr]{%
2	13.3230420273002\\
4	8.37838083761204\\
6	2.71949407809688\\
8	1.25216973095691\\
10	0.882334456208134\\
12	0.810446704423751\\
14	0.590563610382663\\
16	0.246960627111951\\
18	0.236693019458051\\
20	0.12005208133869\\
};
\addlegendentry{IRKA}

\addplot [color=mycolor3,line width=\lineWidth,mark=diamond*,mark repeat={1},mark options={solid}]
  table[row sep=crcr]{%
2	19.8421343335966\\
4	16.6517295225129\\
6	5.34413587378438\\
8	3.9101266045848\\
10	2.43844272290461\\
12	2.05894922489081\\
14	1.25827974777933\\
16	1.23402303131055\\
18	0.588734041079704\\
20	0.354533910218014\\
};
\addlegendentry{pHIRKA ($Q$)}

\addplot [color=mycolor3, dashed,line width=\lineWidth,mark=diamond*,mark repeat={1},mark options={solid}]
  table[row sep=crcr]{%
2	17.5739810465858\\
4	13.1334993919912\\
6	12.1124054595305\\
8	1.27866755398429\\
10	0.904417661219606\\
12	0.830662794905894\\
14	0.745281790898985\\
16	0.703222788458966\\
18	0.287185946832489\\
20	0.305113973870203\\
};
\addlegendentry{pH-IRKA ($\Xmin$)}

\addplot [color=mycolor4,line width=\lineWidth,mark=triangle*,mark repeat={1},mark options={solid}]
  table[row sep=crcr]{%
2	16.9784670971937\\
4	10.4707850325043\\
6	4.2062485114971\\
8	2.60786261227396\\
10	2.33416929737013\\
12	2.29416129765307\\
14	1.30501105877078\\
16	1.30949424287513\\
18	0.719306693895128\\
20	0.667951543029351\\
};
\addlegendentry{spectralFactor ($Q$)}

\addplot [color=mycolor4, dashed,line width=\lineWidth,mark=triangle*,mark repeat={1},mark options={solid}]
  table[row sep=crcr]{%
2	18.883093478458\\
4	8.80096783532217\\
6	8.7428765311436\\
8	1.2819657344551\\
10	1.00977579559012\\
12	0.819958480059919\\
14	0.736666168306758\\
16	0.357651638421531\\
18	0.282757743855276\\
20	0.280125817520025\\
};
\addlegendentry{spectralFactor ($\Xmin$)}

\end{axis}
\end{tikzpicture}%
	\caption{$\mathcal{H}_2$ error for the poroelasticity model with artificial damping parameter $\eta=\num{1e-03}$}
	\label{fig:poro-n980-H2error}
\end{figure}
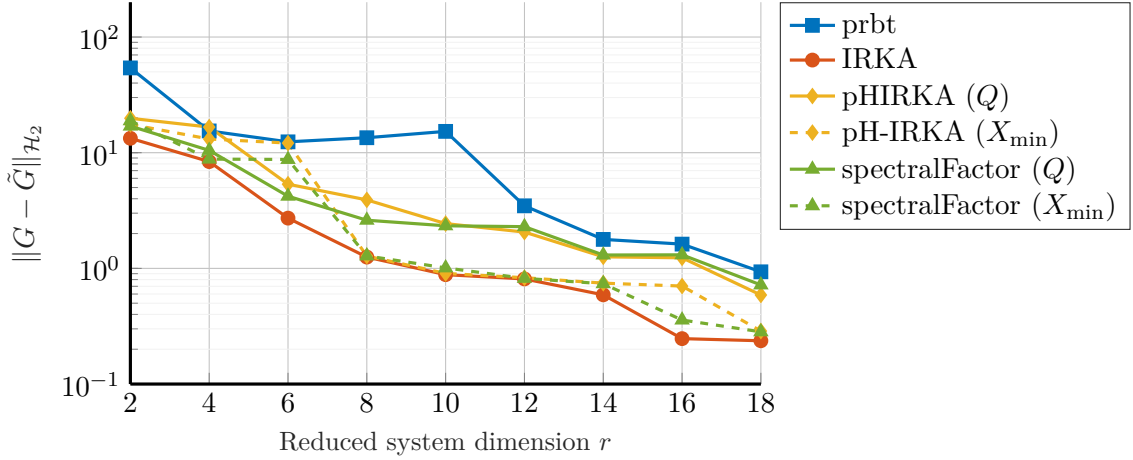
%
%
We notice that for $r=6$, pH-IRKA and our method produce slightly better approximations if using the original pH formulation instead of the pH realization corresponding to $\Xmin$.  Comparing the approximation quality for the spectral factors $\Sigma_H$, depicted in Figure~\ref{fig:poro-n980-spectralFactor-H2Error}, we observe the expected smaller error for the spectral factor corresponding to $\Xmin$. We emphasize that it is not clear how the different numerical errors (structure-preserving minimal realization, solution of the ARE including the artificial feedthrough term, Cholesky-like factorization for potentially indefinite matrices, solution of the Lyapunov equation) contribute to the overall approximation error. A detailed analysis is subject to further research.
%
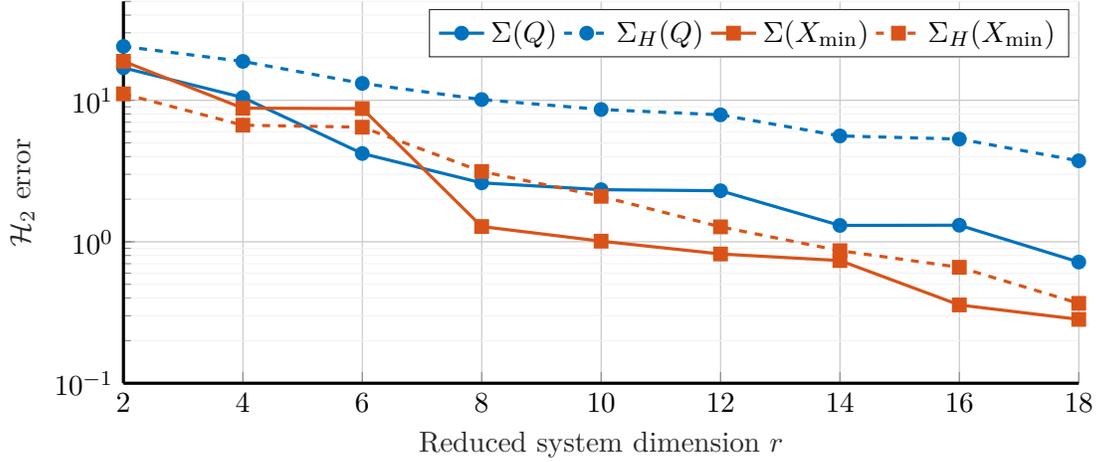
\begin{figure}
	\centering
%
%
\begin{tikzpicture}

\begin{axis}[%
width=5in,
height=2in,
at={(1.011in,0.642in)},
scale only axis,
grid=both,
grid style={line width=.1pt, draw=gray!10},
major grid style={line width=.2pt,draw=gray!50},
axis lines*=left,
axis line style={line width=\lineWidth},
xmin=2,
xmax=18,
xlabel style={font=\color{white!15!black}},
xlabel={Reduced system dimension $r$},
ymode=log,
ymin=1e-01,
ymax=5e1,
yminorticks=true,
ylabel style={font=\color{white!15!black}},
ylabel={ $\mathcal{H}_2$ error},
axis background/.style={fill=white},
mark size=2.2pt,
legend style={%
	legend cell align=left, 
	align=left, 
	draw=white!15!black,
	at={(0.99,0.98)},
	anchor=north east,
},
legend columns=-1,
]
\addplot [color=mycolor1,line width=\lineWidth,mark=*,mark repeat={1}]
  table[row sep=crcr]{%
2	16.9784670971937\\
4	10.4707850325043\\
6	4.2062485114971\\
8	2.60786261227396\\
10	2.33416929737013\\
12	2.29416129765307\\
14	1.30501105877078\\
16	1.30949424287513\\
18	0.719306693895128\\
20	0.667951543029351\\
};
\addlegendentry{$\Sigma(Q)$}

\addplot [color=mycolor1, dashed,line width=\lineWidth,mark=*,mark repeat={1},mark options={solid}]
  table[row sep=crcr]{%
2	24.0432015473772\\
4	18.8468534414091\\
6	13.1738373518466\\
8	10.1226900768408\\
10	8.61383549713323\\
12	7.90016820979057\\
14	5.5995696591065\\
16	5.32262375928714\\
18	3.73567804652807\\
20	3.05834682010553\\
};
\addlegendentry{$\Sigma_H(Q)$}

\addplot [color=mycolor2,line width=\lineWidth,mark=square*,mark repeat={1},mark options={solid}]
  table[row sep=crcr]{%
2	18.883093478458\\
4	8.80096783532217\\
6	8.7428765311436\\
8	1.2819657344551\\
10	1.00977579559012\\
12	0.819958480059919\\
14	0.736666168306758\\
16	0.357651638421531\\
18	0.282757743855276\\
20	0.280125817520025\\
};
\addlegendentry{$\Sigma(\Xmin)$}

\addplot [color=mycolor2, dashed,line width=\lineWidth,mark=square*,mark repeat={1},mark options={solid}]
  table[row sep=crcr]{%
2	11.1117716908951\\
4	6.67006455799772\\
6	6.4590999509257\\
8	3.13895999272374\\
10	2.09552172986841\\
12	1.27491193252792\\
14	0.865129878293167\\
16	0.660489240918842\\
18	0.367105164588867\\
20	0.257205004529665\\
};
\addlegendentry{$\Sigma_H(\Xmin)$}

\end{axis}
\end{tikzpicture}%
	\caption{Comparison of the $\mathcal{H}_2$ error for the system $\Sigma$ and its spectral factor~$\Sigma_H$}
	\label{fig:poro-n980-spectralFactor-H2Error}
\end{figure}

\section{Summary}

We have presented a novel structure-preserving model reduction method, namely \Cref{alg:MORspectralFactor}, that retains passivity or, equivalently, a port-Hamiltonian structure, within the reduced-order model. Our algorithm exploits the connection of passivity to the spectral factorization of the Popov function via the KYP inequality. We have shown (cf.~\Cref{thm:HankelSingularValuesSpectralFactor}) that the minimal solution of the KYP inequality, respectively, the associated Lur'e and algebraic Riccati equation, is preferable from a model reduction perspective. This not only applies to our methodology but translates to the structure-preserving variant of the iterative rational Krylov algorithm (IRKA), called pH-IRKA. While the corresponding pH realization is in general not sparse, the spectral factor retains the sparsity pattern of the original system independent of the particular solution of the KYP inequality, allowing for MOR in a large-scale context. Moreover, since we can use any stability-preserving MOR method for the spectral factor, existing MOR methods can be used to construct the passive ROM. The numerical examples demonstrate that our algorithm can produce accurate low-dimensional passive surrogate models whose $\mathcal{H}_2$ error is close to $\mathcal{H}_2$-optimal methods (which do not guarantee passivity).

As future work, we plan to investigate the impact of the different numerical approximation errors that impact our ROM construction. Moreover, an extension to passive descriptor systems, respectively, port-Hamiltonian differential-algebraic systems, is interesting.

\subsection*{Acknowledgments} T. Breiten acknowledges funding from the DFG within the
Sonderforschungsbereich/Transregio 154 ``Mathematical Modelling, Simulation and Optimization
using the Example of Gas Networks''.
 B. Unger acknowledges funding from the DFG under Germany's Excellence Strategy -- EXC 2075 -- 390740016 and is thankful for support by the Stuttgart Center for Simulation Science (SimTech).

\bibliographystyle{siam}
\bibliography{ph_mor}

\end{document}